\documentclass[a4paper,10pt,reqno]{amsart}
\usepackage[english]{babel}
\usepackage[latin1]{inputenc}
\usepackage[T1]{fontenc}
\usepackage{amssymb,amsthm,amsmath}
\usepackage{mathtools}
\usepackage{mathrsfs}
\usepackage{mathabx} 
\usepackage{tikz}
\usepackage{cite}
\usepackage{enumerate}
\usepackage{color}
\usepackage{geometry}
\usepackage{hyperref}
\theoremstyle{plain} 
\newtheorem{theorem}{Theorem}[section]
\newtheorem{lemma}{Lemma}[section]

\theoremstyle{definition}

\theoremstyle{remark}
\newtheorem{remark}{Remark}[section]

\DeclarePairedDelimiter{\abs}{\lvert}{\rvert} 
\DeclarePairedDelimiter{\norm}{\lVert}{\rVert}

\DeclareMathOperator{\supp}{supp}
\DeclareMathOperator*{\essupp}{ess\,sup}

\newcommand{\invf}[1]{\,\, \widecheck{#1}}
\newcommand{\invt}[1]{\,\, \widecheck{#1}\,\,^{\tau}}
\newcommand{\invx}[1]{\,\, \widecheck{#1}\,\,^{\xi}}
\newcommand{\invet}[1]{\,\, \widecheck{#1}\,\,^{\eta \tau}}
\newcommand{\dirf}[1]{\,\, \widehat{#1}}
\newcommand{\ft}[1]{\,\, \widehat{#1}\,\,^{\tau}}
\newcommand{\fxy}[1]{\,\, \widehat{#1}\,\,^{x y}}
\newcommand{\fyt}[1]{\,\, \widehat{#1}\,\,^{y t}}
\newcommand{\R}{\mathbb{R}}
\newcommand{\N}{\mathbb{N}}

\newcommand{\normeq}[1]{{\left\vert\kern-0.25ex\left\vert\kern-0.25ex\left\vert #1 
    \right\vert\kern-0.25ex\right\vert\kern-0.25ex\right\vert}}

\newcommand{\Lt}{{L^2(\R^2\times[0,1])}}

\newenvironment{system}%
{\left\lbrace\begin{array}{@{}l@{}}}%
{\end{array}\right.}

\title[Unique continuation for Z-K equation]{Uniqueness results for Zakharov-Kuznetsov equation}

\author{Lucrezia Cossetti}
\author{Luca Fanelli}
\author{Felipe Linares}
\address{Lucrezia Cossetti: BCAM - Basque Center for Applied Mathematics, Mazarredo, 14 E48009 Bilbao, Basque Country, Spain}
\address{Luca Fanelli: Dipartimento di Matematica, Sapienza Università di Roma, P. le A. Moro 5, 00185, Roma}
\address{Felipe Linares: IMPA, Instituto Matem\'atica Pura e Aplicada, Estrada Dona Castorina 110, 22460-320, Rio de Janeiro, RJ,
Brazil} 
\email{lcossetti@bcamath.org}
\email{fanelli@mat.uniroma1.it}
\email{linares@impa.br}

\subjclass[2010]{35Q35, 35Q53}

\keywords{Dispersive equations; Unique continuation property}

\begin{document}

\date{\today}


\begin{abstract}
	In this paper we study uniqueness properties of solutions to the Zakharov-Kuznetsov equation of plasma physic.
	
Given two sufficiently regular solutions $u_1, u_2,$ we prove that, if $u_1-u_2$ decays  fast enough at two distinct times, then $u_1\equiv u_2.$   
\end{abstract}

\maketitle


\section{Introduction}
This paper is concerned with uniqueness properties results for solutions of  the so called Zakharov-Kuznetsov equation
\begin{equation}\label{Z-K}
	\partial_t u + \partial_x^3 u + \partial_x\partial_y^2 u + u \partial_x u=0, \qquad (x,y)\in \R^2, \quad t\in [0,1].
\end{equation}
Equation~\eqref{Z-K} is one of the variants of the $(2+1)$-dimensional generalization of the Korteweg-de Vries (KdV) equation that reads
\begin{equation}\label{KdV}
	\partial_t u + \partial_x^3 u + u \partial_x u=0, \qquad x\in \R, \quad t\in [0,1]. 
\end{equation}
The equation was introduced in the context of plasma physic by Zakharov and Kuznetsov in~\cite{Z_K}, where they formally deduced that the propagation of nonlinear ion-acoustic waves in magnetized plasma is governed by this mathematical model. A rigorous derivation of equation~\eqref{Z-K} was given by Lannes, Linares and Saut in~\cite{L_L_S}.

\medskip
The problem of local and global well-posedness for the Cauchy problem associated to~\eqref{Z-K} has extensively been studied. Up to date the best local well-posedness result available in the literature was obtained independently by Molinet and Pilod~\cite{M_P} and Gr\"unrock and Herr~\cite{G_H} for initial data in $H^s(\R^2),$ $s>\frac{1}{2}.$
Then the global theory follows by standard arguments based on $L^2$ and $H^1$ conservation laws. We refer to~\cite{F, L_P, L_P_S, L_PII} and references therein for other results of this type and several additional remarks concerning with properties of this equation.

\medskip
Our main goal is to prove uniqueness properties from two distinct times for equation~\eqref{Z-K}. More precisely we want to deduce sufficient conditions on the behavior of the difference $u_1-u_2$ of two solutions $u_1,$ $u_2$ of~\eqref{Z-K} at two different times, $t_0=0$ and $t_1=1,$ which guarantee  that  $u_1\equiv u_2.$ This kind of results is inspired to the program performed in~\cite{E_K_P_V,E_K_P_V1,E_K_P_V2,E_K_P_V3,E_K_P_V4,E_K_P_V5} for Schr\"odinger and KdV (see also~\cite{S_S} and Remark~\ref{rk:linear_problem} below for further details). 

\medskip
The main motivation for our study is a recent work by Bustamante, Isaza and Mej\'ia~\cite{B_I_M} where an upper bound for the possible decay at two different times of a non-trivial difference of two solutions of~\eqref{Z-K} was given. 
More precisely they prove the following:
\begin{theorem}[\cite{B_I_M}~]\label{thm:B_I_M}
	Suppose that for some small $\varepsilon>0$
	\begin{equation*}
		u_1, u_2 \in C\big([0,1]; H^4(\R^2) \cap L^2((1+ x^2 + y^2)^{\frac{4}{3} + \varepsilon}\, dx dy)\big) \cap C^1 \big( [0,1];L^2(\R^2) \big),
	\end{equation*}
		are solutions of~\eqref{Z-K}. Then there exists a universal constant $a_0>0,$ such that if for some $a>a_0$
		\begin{equation*}
			u_1(0)-u_2(0), u_1(1)-u_2(1) \in L^2(e^{a(x^2 + y^2)^{3/4}}\, dx dy),
		\end{equation*}
		then $u_1\equiv u_2.$
\end{theorem}

\begin{remark}\label{rk:linear_problem}
As the authors  in \cite{B_I_M} pointed out, this result does not appear to be optimal, indeed the symmetric character in $x$ and $y$ of the decay assumption does not reflect the non symmetric form, with respect to $x$ and $y,$ of equation~\eqref{Z-K}. 

To explain this fact, let us sketch the analog picture for KdV and Schr\"odinger equations. As regards with the KdV equation, Escauriaza, Kenig, Ponce and Vega in~\cite{E_K_P_V2}, considering $u_1$ and $u_2$ two solutions of
\begin{equation}\label{eq:KdV}
	\partial_t u + \partial_x^3 u + u \partial_x u=0, \qquad (x,t)\in \R \times [0,1],
\end{equation} 
deduced  that  there exists a universal constant $a_0>0$ such that if for some $a>a_0$
\begin{equation}\label{KdV_decay}
u_1(0)-u_2(0),\quad u_1(1)-u_2(1)\in L^2(e^{a x_+^{3/2}}\, dx),
\end{equation}
then $u_1\equiv u_2.$ (Here $x_+:=\max\{x;0\}$).

Instead in~\cite{E_K_P_V5} the same authors considered solutions of the Schr\"odinger equation
\begin{equation}\label{eq:Schroedinger}
	\partial_t u= i(\Delta u + V(x,t) u), \qquad (x,t)\in \R^n \times [0,1], 
\end{equation}
and proved that if $u$ is a solution of this equation and if there are two positive constants $\alpha$ and $\beta$ with $\alpha \beta <4$ such that
\begin{equation}\label{Schroedinger_decay}
	\norm{e^{\abs{x}^2/\beta^2} u(0)}_{L^2(\R^n)},\quad \norm{e^{\abs{x}^2/\alpha^2} u(1)}_{L^2(\R^n)}<\infty, 
\end{equation}
then $u\equiv 0.$

The  value $3/2$ in the exponent in~\eqref{KdV_decay} arises in the asymptotic behavior of the Airy function, while the Gaussian decay is known to be the sharpest possible simultaneous decay for both a function $f$ and its Fourier transform $\widehat{f},$ which explains~\eqref{Schroedinger_decay} together with the aid of the explicit formula for the Schr\"odinger kernel. 

	\end{remark}
	
	\medskip
	
	For the ZK equation one might expect to have a sharp decay of the form  $e^{-a x^{3/2} -b y^2}$. This is because of  
	the decay of the fundamental solution of KdV and the Gaussian parabolic heritage arising from the Fourier uncertainty.

\medskip
Recently, Faminskii and Antonova in~\cite{F_A} showed that the previous ``natural'' ansatz for the decay assumption is wrong, they proved that the fundamental solution to the operator $\partial_t + \partial_x^3 + \partial_x\partial_y^2$ still displays an exponential decay but just in the $x$ variable. More precisely, considering the IVP
\begin{equation*}
	\begin{system}
		\partial_t u + \partial_x^3 u + \partial_x \partial_y^2 u=0\\
		u(x,y,0)=u_0(x,y)
	\end{system}
\end{equation*}    
whose solution given as a convolution by
\begin{equation*}
	u(x,y,t)= \frac{\theta(t)}{t^\frac{2}{3}} S\Big( \frac{x}{t^\frac{1}{3}},\frac{y}{t^\frac{1}{3}} \Big) \ast u_0(x,y), 
\end{equation*}
where 
\begin{equation}\label{def:S}
	S(x,y):= \frac{1}{2\pi} \mathcal{F}^{-1}\big[(\xi,\eta)\mapsto e^{i(\xi^3 + \xi \eta^2)}\big]= \frac{1}{4\pi^2} \int_{\R^2} e^{i\xi x + i \eta y} e^{i(\xi^3 +  \xi \eta^2)}\, d\xi d\eta,
\end{equation}
$\theta$ is the Heaviside function and $\mathcal{F}^{-1}$ represents the inverse Fourier transform, they prove for the function $S$ the following result.

\begin{lemma}
	Let $S(x,y)$ be as in~\eqref{def:S}, for any $x\in \R$ and integer $k\geq 0$ the derivative $\partial_x^k S(x,y)$ belongs to the Schwartz space $\mathcal{S(\R)}$ with respect to $y$ and there exists a constant $c_0>0$ such that for any $x_0\in \R,$ integer $m\geq 0$ and multi-index $\nu$
	\begin{equation}\label{right_decay}
		(1 + \abs{y})^m \abs{\partial_{x,y}^\nu\, S(x,y)}\leq c(m, \abs{\nu}, x_0) e^{-c_0(x-x_0)^{3/2}}\quad \forall\, x\geq x_0, \, \forall\, y \in \R.
	\end{equation}
\end{lemma}

This lemma suggests what should be the sharp decay for solutions to the nonlinear problem. Our main result in this work shows that is in fact the case. More precisely we prove the following:

\begin{theorem}\label{main_result}
	Suppose that for some small $\varepsilon>0,$ 
	\begin{equation}\label{hypotheses}
		u_1, u_2 \in C\big( [0,1]; H^4(\R^2) \cap L^2((1+\abs{x})^{2(\frac{4}{3} + \varepsilon)} \, dx dy) \big) \cap C^1( [0,1]; L^2(\R^2)),
	\end{equation}	
		are solutions of the equation~\eqref{Z-K}.
		
		Then there exists a universal constant $a_0,$ such that if for some $a>a_0$
		\begin{equation}\label{decay_properties}
			u_1(0)-u_2(0), u_1(1)-u_2(1) \in L^2\big( e^{a \abs{x}^{3/2}} dx dy\big), 
		\end{equation}
		then $u_1 \equiv u_2.$
\end{theorem}

In order to obtain this result, following~\cite{G_H} we introduce the linear change of variables 
\begin{equation}\label{change_variable}
	\begin{system}
		x=\frac{x'+y'}{2\mu} \vspace{0.2cm} \\
		y=\frac{x'-y'}{2\lambda}
	\end{system}
\end{equation}
with $\lambda= \sqrt{3} \mu$ and $\mu=4^{-1/3}.$

It turns out that if $u=u(x,y,t)$ solves~\eqref{Z-K} then $\widetilde{u}(x',y',t):=u\big(\frac{x'+y'}{2\mu},\frac{x'-y'}{2\lambda},t\big)$ solves
\begin{equation}\label{Z-K_s}
	\partial_t u + (\partial_{x}^3 + \partial_{y}^3) u + 4^{-1/3} u \,(\partial_x + \partial_y)u=0, \qquad (x,y)\in \R^2, \quad t\in [0,1],
\end{equation}
that is a symmetric version of Z-K equation~\eqref{Z-K}. Here with abuse of notation we have called $x',y',\widetilde{u}$ as $x,y,u$ respectively.

\medskip

Therefore Theorem~\ref{main_result} is a consequence of the following result.
   
\begin{theorem}\label{main_result_s}
Suppose that for some $\varepsilon>0,$ 
	\begin{equation}\label{hypotheses_s}
		u_1, u_2 \in C\big( [0,1]; H^4(\R^2) \cap L^2((1+\abs{x + y})^{2(\frac{4}{3} + \varepsilon)} \, dx dy) \big) \cap C^1( [0,1]; L^2(\R^2)),
	\end{equation}	
		are solutions of the equation~\eqref{Z-K_s}.
		
		Then there exists a universal constant $a_0,$ such that if for some $a>a_0$
		\begin{equation}\label{decay_properties_s}
			u_1(0)-u_2(0), u_1(1)-u_2(1) \in L^2\big( e^{a \abs{x+y}^{3/2}} dx dy\big), 
		\end{equation}
		then $u_1 \equiv u_2.$
\end{theorem}

Notice that, if $u_1,u_2$ solve~\eqref{Z-K_s}, then $v:=u_1-u_2$ is a solution to
\begin{equation}\label{eq:difference}
	\partial_t v + (\partial_x^3 + \partial_y^3) v + 4^{-\frac{1}{3}} u_1 (\partial_x + \partial_y) v + 4^{-\frac{1}{3}}(\partial_x + \partial_y)u_2\, v=0.
\end{equation}
Since it comes into play in the proof of Theorem~\ref{main_result_s} and we think it is of independent interest, we state the following linear result for~\eqref{eq:difference} (actually it is considered a slightly more general equation than~\eqref{eq:difference}).

\begin{theorem}\label{linear_result}
	Suppose that for some small $\varepsilon>0,$ 
	\begin{equation*}
		v\in C\big( [0,1]; H^3(\R^2) \cap L^2((1+\abs{x + y})^{2(\frac{4}{3} + \varepsilon)} \, dx dy) \big) \cap C^1( [0,1]; L^2(\R^2)),
	\end{equation*}
	is a solution of
	\begin{equation}\label{generalized_Z-K_diff}
		\partial_t v + (\partial_x^3 + \partial_y^3) v + a_1(x,y,t)(\partial_x + \partial_y) v + a_0(x,y,t)v=0,
	\end{equation}
	where $a_0 \in L^\infty \cap L_x^2L_{y,t}^\infty $ and $a_1 \in L^\infty \cap L_x^2 L_{y t}^\infty \cap L_x^1 L_{y t}^\infty.$
	
	Then there exists a universal constant $a_0>0$ such that if for some $a>a_0$
	\begin{equation*}
		v(0), v(1) \in L^2(e^{a \abs{x+y}^{3/2}} \, dx dy),
	\end{equation*}
	then $v\equiv 0.$
\end{theorem}

We shall see that, under the hypotheses of Theorem~\ref{main_result_s},~\eqref{eq:difference} turns out to be a particular case of~\eqref{generalized_Z-K_diff} with $a_0=4^{-\frac{1}{3}}(\partial_x+\partial_y)u_2$ and $a_1=4^{-\frac{1}{3}} u_1,$ therefore our result in Theorem~\ref{main_result_s} will follow as a consequence of the validity of Theorem~\ref{linear_result}.

\newpage
The optimality of Theorem~\ref{main_result_s} (and thus of Theorem~\eqref{main_result}) is proved in the following theorem.
\begin{theorem}\label{thm:optimality}
	Suppose that for some $\varepsilon>0$
	\begin{equation*}
		u \in C\big( [0,1]; H^4(\R^2) \cap L^2((1+\abs{x + y})^{2(\frac{4}{3} + \varepsilon)} \, dx dy) \big) \cap C^1( [0,1]; L^2(\R^2)),
	\end{equation*}
	is a solution of the equation~\eqref{Z-K_s}. Let $a_0$ be a positive constant such that 
	\begin{equation*}
		u(0)\in L^2(e^{a_0(x+y)_+^{3/2}}\, dx dy),
	\end{equation*}
	then $u$ satisfies
	\begin{equation*}
		\sup_{t\in [0,1]} \int_{\R^2} e^{a(t) (x+y)_+^{3/2}}\abs{u(x,y,t)}^2\, dx dy\leq C=C(a_0, \norm{u}_{C([0,1]; H^2(\R^2))}, \norm{e^{a_0(x+y)_+^{3/2}/2} u}_{C([0,1];H^2(\R^2))}),
	\end{equation*}
	with
	\begin{equation*}
		a(t)= \frac{a_0}{(1+ 27 a_0^2 t/2)^{1/2}}.
	\end{equation*}
\end{theorem}  

\begin{remark}
	Notice that following the argument in~\cite{E_K_P_V2} it can be proved that given $a_0>0$ and $\delta>0$ there exists a non trivial initial datum $u_0\in \mathbb{S}(\R^2),$ $c_1, c_2>0$ and an interval of time $\Delta T>0$ such that the corresponding solution $u(x,y,t)$ of~\eqref{Z-K_s} with initial datum $u_0:=u(0)$ satisfies
	\begin{equation*}
		c_1\,e^{-(a_0 + \delta) (x+y)^{3/2}}\leq u(x,y,t)\leq c_2\, e^{-(a_0 - \delta) (x+y)^{3/2}}, \qquad x+y> > 1,\quad t\in [0, \Delta T]. 
	\end{equation*}
\end{remark}

\bigskip
The paper is organized as follows. Section 2 represents the core of our work as regards to the symmetric Z-K, here we are concerned with the proofs of Theorem~\ref{linear_result} and of its nonlinear counterpart Theorem~\ref{main_result_s}. In order to do that, following the scheme in~\cite{E_K_P_V2}, we introduce two types of estimates, a lower bound which follows after performing a suitable Carleman estimate and an upper bound for the $H^2$ norm of the solutions which exploits the exponential decay assumed for the initial and final data.

In Section 3, our main result Theorem~\ref{main_result} for the original Z-K equation is proved.

Finally in Section 4 the proof of Theorem~\ref{thm:optimality} is given.

\bigskip
\subsection*{Acknowledgment}
The authors would like to thank Pedro Caro for helpful comments and suggestions.

The research of L.C. is supported by the Basque Government through the BERC 2014-2017 program and by Spanish Ministry of Economy and Competitiveness MINECO: BCAM Severo Ochoa excellence accreditation SEV-2013-0323.


\section{Proof of Theorems~\ref{main_result_s}-~\ref{linear_result}}
This section is concerned with the proof of Theorems~\ref{main_result_s}-~\ref{linear_result}. As sketched above, we will prove a lower and an upper bound in suitable weighted norms for the solution $v$ to~\eqref{generalized_Z-K_diff} and then perform a contradiction argument.

\subsection{Lower bound}
This subsection is mainly interested in the proof of the following result.

\begin{theorem}\label{lower_bound}
	Let $v \in C([0,1]; H^3(\R^2))$ be a solution of~\eqref{generalized_Z-K_diff} with $a_0, a_1 \in L^\infty(\R^3).$ Assume that 
	\begin{equation*}
		\int_{\R^2} \int_0^1 \big(\abs{v}^2 + \abs{\nabla v}^2 + \abs{\Delta v}^2\big)\, dx\,dy\,dt \leq A^2.
	\end{equation*}
	Let $\delta>0,$ $r \in (0,\frac{1}{2})$ and $Q:= \{ (x,y,t)\colon \sqrt{x^2 + y^2}\leq 1,\, t \in [r, 1-r]\}$ and suppose that $\norm{v}_{L^2(Q)}\geq \delta.$ Then there exist constants $\widetilde{R}_0, c_0, c_1$ depending on $A, \norm{a_0}_{\infty}$ and $\norm{a_1}_{\infty},$ such that for $R\geq \widetilde{R}_0$
	\begin{equation*}
		A_R(v):= \Big( \int_0^1 \int_{Q_R} \big(\abs{v}^2 + \abs{\nabla v}^2 + \abs{\Delta v}^2\big)\, dx\, dy\, dt \Big)^\frac{1}{2}\geq c_0 e^{- c_1 R^\frac{3}{2}},
	\end{equation*}
	where $Q_R:=\{ (x, y) \colon R-1\leq \abs{x+y}\leq R \,\land\, R-1\leq \abs{x-y}\leq R \}.$
\end{theorem}	

\begin{figure}[ht]
\centering
\begin{tikzpicture}[scale=.4]
\draw[domain=0:5,color=gray,samples=2]
plot(\x, {-\x +5});
\draw[domain=0:5,color=gray,samples=2]
plot(\x, {\x -5});
\draw[domain=-5:0,color=gray,samples=2]
plot(\x, {-\x -5});
\draw[domain=-5:0,color=gray,samples=2]
plot(\x, {\x +5});
\draw[domain=0:6,color=gray,samples=2]
plot(\x, {-\x +6});
\draw[domain=0:6,color=gray,samples=2]
plot(\x, {\x -6});
\draw[domain=-6:0,color=gray,samples=2]
plot(\x, {-\x -6});
\draw[domain=-6:0,color=gray,samples=2]
plot(\x, {\x +6});
\filldraw[gray] (0,5)-- (0,6)--(6,0)--(5,0);
\filldraw[gray] (0,-5)-- (0,-6)--(6,0)--(5,0);
\filldraw[gray] (0,-5)-- (0,-6)--(-6,0)--(-5,0);
\filldraw[gray] (0,5)-- (0,6)--(-6,0)--(-5,0);
\coordinate [label=above left: \textcolor{black}{$\scriptstyle R-1$}] (a) at (4.4,0); 
\fill[black] (4.93,0) circle (3pt);
\coordinate [label=above right: \textcolor{black}{$\scriptstyle R$}] (b) at (5.7,0); 
\fill[black] (6,0) circle (3pt);
\draw[thick,->](-7,0)--(7.4,0) node[below]{$x$};
\draw[thick,->](0,-7)--(0,7.4) node[right]{$y$};
\end{tikzpicture}
\caption{The region $Q_R$} 
\end{figure}
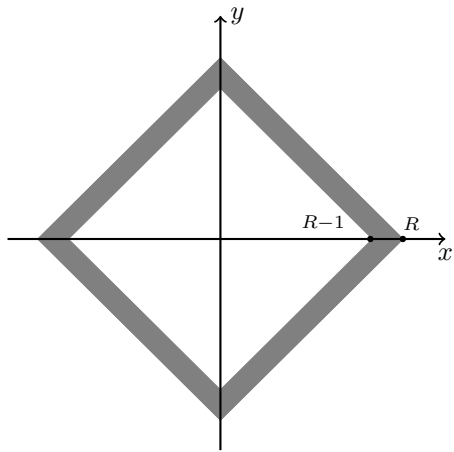

\newpage
The previous idea of establishing lower bounds for the asymptotic behavior of a suitable norm of the solution in an annulus domain  stems from a work by Bourgain and Kenig~\cite{B_K} on a class of stationary Schr\"odinger operators $-\Delta + V(x)$ in which the property of spectral localization, that is the phenomenon for which the point spectrum of the analyzed operator presents exponentially decaying eigenfunctions, is studied.

In that work they needed precise quantitative information on the rate of local vanishing for eigenfunctions, more precisely, local bounds on the eigenfunctions both from above and from below were required. Unlike the upper bound, which just needs classical tools to be achieved, the lower bound is a more subtle issue.
The statement (Lemma 3.10 in~\cite{B_K}) is as follows.
\begin{lemma}
	Let $u$ be a bounded solution of $\Delta u + Vu=0$ in $\R$ with suitable additional assumptions about $V.$ Let $x_0\in \R^n, \abs{x_0}=R>1.$ Then
	\begin{equation*}
		\max_{R-1<\, x<\, R}\, \abs{u(x)} >c_0 e^{-c_1 (\log R) R^\frac{4}{3}}.
	\end{equation*}
\end{lemma}

 This was derived from the following Carleman type estimate.
\begin{lemma}\label{Bourgain_Kenig}
	There are constants $C_1, C_2, C_3,$ depending only on $n$ and an increasing function $w=w(r)$ for $0<r<10$ such that
	\begin{equation*}
		\frac{1}{C_1}< \frac{w(r)}{r}<C_1
	\end{equation*}
	and for all $f\in C_0^\infty(B_{10}\setminus \{0\}),\, \alpha > C_2,$ we have
	\begin{equation*}
		\alpha^3 \int_{\R^d} w^{-1-2\alpha} f^2 \leq C_3 \int_{\R^d} w^{2-2\alpha} (\Delta f)^2.
	\end{equation*}
\end{lemma}

\medskip
In order to obtain the lower bound in Theorem~\ref{lower_bound}, in the same spirit as Bourgain and Kenig we will start performing a Carleman estimate for our operator
\begin{equation}\label{operator_difference}
	P=\partial_t + (\partial_x^3 + \partial_y^3) + a_1(x,y,t)(\partial_x + \partial_y) + a_0(x,y,t),
\end{equation}
where $a_0,a_1\in L^\infty(\R^3).$

As a starting point we will prove the following Carleman estimate for the leading part of the operator $P,$ namely $\partial_t + \partial_x^3 + \partial_y^3.$

\begin{lemma}\label{Carleman1}
	Assume that $\varphi\colon [0,1] \to \R$ is a smooth function. Then, there exist two constants $c>0$ and $M_1=M_1(\norm{\varphi'}_{\infty}, \norm{\varphi''}_{\infty})>0$ such that the inequality
	\begin{multline}\label{Carleman}
		\frac{\alpha^\frac{5}{2}}{R^3} \norm*{e^{\alpha \phi(x,y,t)} \phi(x,y,t)\, g}_\Lt + \frac{\alpha^\frac{3}{2}}{R^2} \norm{ e^{\alpha \phi(x,y,t)} \abs{\nabla g}}_\Lt\\
		\leq c \norm{ e^{\alpha \phi(x,y,t)}(\partial_t + \partial_x^3 + \partial_y^3)g}_\Lt
	\end{multline}
	holds, for $R\geq 1,$ $\alpha$ such that $\alpha^2\geq M_1 R^3,\, g \in C^\infty_0(\R^2 \times [0,1])$ supported in
	\begin{equation*}
		\Big\{ (x,y,t) \in \R^2 \times[0,1]\colon \abs*{\frac{\varkappa}{R} + \varphi(t) \xi}\geq 1  \Big\}
	\end{equation*}
	and $\phi(x,y,t)= \abs*{\frac{\varkappa}{R} + \varphi(t) \xi}^2=\big( \frac{x}{R} + \varphi(t) \big)^2 + \big( \frac{y}{R} + \varphi(t) \big)^2,$ with $\varkappa=(x,y)$ and $\xi=(1,1).$
\end{lemma}

\begin{remark}\label{parameter_alpha}
	In order to obtain from the previous result an estimate involving the whole operator $P$ as defined in~\eqref{operator_difference}, a very essential role is played by the multiplicative parameter $\alpha$ on the left-hand side of~\eqref{Carleman}. Indeed by taking $\alpha$ sufficiently large, we can make the term on the left-hand side as large as we need in order to absorb potential lower order terms.

 This fact can be seen at work explicitly in the proof of Lemma~\ref{Carleman2} below, where a Carleman estimate for the whole operator $P=\partial_t + \partial_x^3 + \partial_y^3 + a_1(\partial_x + \partial_y) + a_0$ is given. Indeed by virtue of the assumption $\alpha^2\geq M_1 R^3$ in Lemma~\ref{Carleman1}, the terms $\frac{\alpha^\frac{5}{2}}{R^3}$ and $\frac{\alpha^\frac{3}{2}}{R^2}$ grow as a positive fractional power of $R,$ therefore, being $R$ a large parameter, it will allow us to include in the estimate lower order derivatives.
\end{remark}

\begin{proof}
		From now on with an abuse of notation we will write $L^2$ instead of $L^2(\R^2 \times [0,1]).$ 
		
	 Because of the difficulty to prove an exponentially weighted estimate, as usual in this context, we reduce ourselves into proving an estimate for the conjugated operator
		\begin{equation*}
			e^{\alpha \phi(x,y,t)}(\partial_t + \partial_x^3 + \partial_y^3) e^{-\alpha \phi(x,y,t)}.
		\end{equation*}
		The main point in the proof is, roughly speaking, a ``positive commutator argument'' that will give a lower bound for the conjugated operator $e^{\alpha \phi} (\partial_t + \partial_x^3 + \partial_y^3)e^{-\alpha \phi}$ once it is decomposed as a sum of its symmetric and skew-symmetric part.
		
		In order to do that we define $f= e^{\alpha \phi(x,y,t)} g,$ observe that 
		\begin{equation*}
			\begin{split}
				e^{2\alpha \phi} \abs{\nabla g}^2&= e^{2\alpha \phi} \big[( e^{-\alpha \phi} \partial_x f -\alpha \partial_x \phi e^{-\alpha \phi} f )^2 + ( e^{-\alpha \phi} \partial_y f-\alpha \partial_y \phi e^{-\alpha \phi} f )^2 \big]\\
																								  &= ( \partial_x f -\alpha \partial_x \phi f)^2 + (\partial_y f-\alpha \partial_y \phi f)^2.
			\end{split}
		\end{equation*}
		Using this identity it is sufficient to prove
		\begin{equation}\label{Carleman_f}
			c\norm{e^{\alpha \phi}(\partial_t + \partial_x^3 + \partial_y^3) e^{-\alpha \phi} f}_{L^2} \geq \frac{\alpha^\frac{5}{2}}{R^3} \norm{\phi f}_{L^2} + \frac{\alpha^\frac{3}{2}}{R^2} \norm{\partial_x f -\alpha \partial_x \phi f}_{L^2} + \frac{\alpha^\frac{3}{2}}{R^2} \norm{\partial_y f -\alpha \partial_y \phi f}_{L^2}.
		\end{equation}
		A straightforward computation gives
		\begin{equation*}
			\begin{split}
				e^{\alpha \phi}(\partial_t + \partial_x^3 + \partial_y^3) e^{- \alpha \phi} f=& - \alpha \partial_t \phi f + \partial_t f - \alpha \partial_x^3 \phi f + 3 \alpha^2 (\partial_x \phi)(\partial_x^2 \phi) f - \alpha^3(\partial_x \phi)^3 f\\
																																													& -3 \alpha \partial_x^2 \phi \partial_x f + 3 \alpha^2 (\partial_x \phi)^2 \partial_x f - 3\alpha \partial_x\phi \partial_x^2f + \partial_x^3 f\\
																																													& -\alpha \partial_y^3\phi f + 3\alpha^2(\partial_y \phi)(\partial_y^2\phi) f -\alpha^3(\partial_y\phi)^3 f -3\alpha \partial_y^2\phi \partial_y f\\
																																													& +3\alpha^2(\partial_y\phi)^2 \partial_y f -3\alpha\partial_y\phi\partial_y^2 f + \partial_y^3 f.
			\end{split}
		\end{equation*}
		We can write this as
		\begin{equation*}
			e^{\alpha \phi}(\partial_t + \partial_x^3 + \partial_y^3) e^{- \alpha \phi} f= A_\alpha f + S_\alpha f,
		\end{equation*}
		where $A_\alpha$ and $S_\alpha$ are respectively skew-symmetric and symmetric operators given by
		\begin{align*}
			&A_\alpha:= \partial_t + \partial_x^3 + \partial_y^3 + 3\alpha^2(\partial_x\phi)^2 \partial_x + 3\alpha^2(\partial_y\phi)^2 \partial_y + 3\alpha^2 (\partial_x\phi)(\partial_x^2\phi) + 3\alpha^2 (\partial_y\phi)(\partial_y^2\phi)\\
			&S_\alpha:= -3\alpha \partial_x(\partial_x\phi \partial_x \cdot) -3\alpha \partial_y(\partial_y\phi \partial_y \cdot) + \big(-\alpha^3 (\partial_x \phi)^3 -\alpha \partial_x^3\phi\big) + \big(-\alpha^3 (\partial_y \phi)^3 -\alpha \partial_y^3\phi\big) -\alpha \partial_t\phi.
		\end{align*}
		
		Thus one gets
		\begin{equation*}
			\begin{split}
				\norm{e^{\alpha \phi(x,y,t)}(\partial_t + \partial_x^3 + \partial_y^3)e^{-\alpha \phi(x,y,t)} f}_{L^2}^2 &=\norm{(A_\alpha + S_\alpha)f}_{L^2}^2\\
																																																							&=\langle(A_\alpha + S_\alpha)f, (A_\alpha + S_\alpha)f \rangle\\
																																																							&=\norm{A_\alpha f}_{L^2}^2 + \norm{S_\alpha f}_{L^2}^2 + \langle A_\alpha f, S_\alpha f \rangle + \langle S_\alpha f, A_\alpha f \rangle\\
																																																							&\geq \langle[S_\alpha, A_\alpha]f, f\rangle.
			\end{split}
		\end{equation*}
		\begin{remark}
			From now on, to save space, we abbreviate $\int:=\iiint_{\R^2 \times [0,1]}$ and omit the arguments of integrated functions. 
		\end{remark}
		
		Now we choose
		\begin{equation}\label{choice_phi}
			\phi(x,y,t)= \abs*{\frac{\varkappa}{R} + \varphi(t) \xi}^2=\Big( \frac{x}{R} + \varphi(t) \Big)^2 + \Big( \frac{y}{R} + \varphi(t) \Big)^2, 
		\end{equation}
		where $\varkappa=(x,y)$ and $\xi=(1,1).$
		
		First choosing~\eqref{choice_phi} and adding and subtracting the terms $\frac{\alpha^3}{R^4} \norm{\partial_x f -\alpha \partial_x \phi f}_{L^2}^2 + \frac{\alpha^3}{R^4} \norm{\partial_y f -\alpha \partial_y \phi f}_{L^2}^2,$ we get
		\begin{align}
			\label{eq1}\tag*{}
			\hspace{-0.8cm}\langle [S_\alpha, A_\alpha]f ,f \rangle=& \frac{18 \alpha}{R^2} \norm{\partial_x^2 f}_{L^2}^2 + \frac{18 \alpha}{R^2} \norm{\partial_y^2 f}_{L^2}^2\\
			\label{eq2}\tag*{$I_1 +I_1^\ast$}
																							 & -\frac{12 \alpha}{R} \int \varphi'(t) (\partial_x f)^2 -\frac{12 \alpha}{R} \int \varphi'(t) (\partial_y f)^2\\ 
			 \label{eq3}\tag*{}
																							 & + \frac{144 \alpha^3}{R^4} \int \Big( \frac{x}{R} + \varphi(t) \Big)^2 (\partial_x f)^2 + \frac{144 \alpha^3}{R^4} \int \Big( \frac{y}{R} + \varphi(t) \Big)^2 (\partial_y f)^2\\
			\label{eq4}\tag*{}
																							 & - \frac{24 \alpha^3}{R^6} \int f^2 + 2\alpha \int (\varphi'(t) f)^2 + 2\alpha \int  \Big( \frac{x}{R} + \varphi(t) \Big) \varphi''(t) f^2\\
	    \label{eq5}\tag*{}
																							 &- \frac{24 \alpha^3}{R^6} \int f^2 + 2\alpha \int (\varphi'(t) f)^2  + 2\alpha \int  \Big( \frac{y}{R} + \varphi(t) \Big) \varphi''(t) f^2\\
			\label{eq6}\tag*{}
																							 & + \frac{48 \alpha^3}{R^3} \int \Big( \frac{x}{R} + \varphi(t) \Big)^2 \varphi'(t) f^2 + \frac{288 \alpha^5}{R^6} \int \Big(\frac{x}{R} + \varphi(t) \Big)^4 f^2\\
			\label{eq7}\tag*{}
																							 & + \frac{48 \alpha^3}{R^3} \int \Big( \frac{y}{R} + \varphi(t) \Big)^2 \varphi'(t) f^2 + \frac{288 \alpha^5}{R^6} \int \Big(\frac{y}{R} + \varphi(t) \Big)^4 f^2\\
			\label{eq8}\tag*{}
																							 & + \frac{\alpha^3}{R^4} \norm{\partial_x f - \alpha \partial_x \phi f}_{L^2}^2 + \frac{\alpha^3}{R^4} \norm{\partial_y f - \alpha \partial_y \phi f}_{L^2}^2\\
		  \label{eq9}\tag*{$I_2 +I_2^\ast$}
																							 & - \frac{\alpha^3}{R^4} \norm{\partial_x f - \alpha \partial_x \phi f}_{L^2}^2 - \frac{\alpha^3}{R^4} \norm{\partial_y f - \alpha \partial_y \phi f}_{L^2}^2.
		\end{align}
		Let us consider~\eqref{eq2}, for
		\begin{equation*}
			\alpha^2 \geq \norm{\varphi'}_{\infty} R^3,
		\end{equation*}
		it follows that
		\begin{equation}\label{I_1I_1s}
			\begin{split}
			I_1 + I_1^\ast& \geq -\frac{12 \alpha}{R} \int \norm{\varphi'}_\infty (\partial_x f)^2 -\frac{12 \alpha}{R} \int \norm{\varphi'}_\infty (\partial_y f)^2\\
										& \geq - \frac{12 \alpha^3}{R^4} \int (\partial_x f)^2  - \frac{12 \alpha^3}{R^4} \int (\partial_y f)^2.
			\end{split}
		\end{equation}
		We compute~\eqref{eq9} using the explicit expression for $\partial_x \phi$ and $\partial_y \phi:$ 
		\begin{equation*}
			\begin{split}
				I_2 + I_2^\ast=& - \frac{\alpha^3}{R^4} \int (\partial_x f)^2 - \frac{4 \alpha^5}{R^6} \int \Big( \frac{x}{R} + \varphi(t) \Big)^2 f^2 + \frac{4 \alpha^4}{R^5} \int \Big( \frac{x}{R} + \varphi(t) \Big) f \partial_x f \\
											 &- \frac{\alpha^3}{R^4} \int (\partial_y f)^2 - \frac{4 \alpha^5}{R^6} \int \Big( \frac{y}{R} + \varphi(t) \Big)^2 f^2 + \frac{4 \alpha^4}{R^5} \int \Big( \frac{y}{R} + \varphi(t) \Big) f \partial_y f.
			\end{split}
		\end{equation*}
		Now let us just consider the last terms in the first and the second rows of the previous identity, using the classical Young inequality
		\begin{equation}\label{Young}
			a\, b \leq \frac{a^p}{p} + \frac{b^q}{q}, \qquad a,b>0, \qquad  \frac{1}{p} + \frac{1}{q}=1,
		\end{equation}
		we obtain
		\begin{equation*}
			\begin{split}
				\frac{4 \alpha^4}{R^5} \int \Big( \frac{x}{R} + \varphi(t) \Big) f \partial_x f+ \frac{4 \alpha^4}{R^5} \int \Big( \frac{y}{R} + \varphi(t) \Big) f \partial_y f 	&\geq - \frac{4 \alpha^4}{R^5} \int \abs*{\frac{x}{R} + \varphi(t)} \abs{f} \abs{\partial_x f} - \frac{4 \alpha^4}{R^5} \int \abs*{ \frac{y}{R} + \varphi(t)} \abs{f} \abs{\partial_y f}\\
				&\geq -\frac{2\alpha^5}{R^6} \int \Big(\frac{x}{R} + \varphi(t)\Big)^2 f^2 - \frac{2\alpha^3}{R^4} \int (\partial_x f)^2\\
				&\phantom{~\geq} -\frac{2\alpha^5}{R^6} \int \Big(\frac{y}{R} + \varphi(t)\Big)^2 f^2 - \frac{2\alpha^3}{R^4} \int (\partial_y f)^2.
				\end{split}
			\end{equation*}
		Since $\abs*{\frac{\varkappa}{R} + \varphi(t) \xi}\geq 1,$ then one obtains
		\begin{equation*}
			\begin{split}
				I_2 + I_2^\ast&\geq -\frac{6\alpha^5}{R^6} \int \abs*{\frac{\varkappa}{R} + \varphi(t)\xi}^2 f^2 - \frac{3\alpha^3}{R^4} \int (\partial_x f)^2 - \frac{3\alpha^3}{R^4} \int (\partial_y f)^2\\
											&\geq  -\frac{6\alpha^5}{R^6} \int \abs*{\frac{\varkappa}{R} + \varphi(t)\xi}^4 f^2 - \frac{3\alpha^3}{R^4} \int (\partial_x f)^2 - \frac{3\alpha^3}{R^4} \int (\partial_y f)^2.
			\end{split}
		\end{equation*}
		Gathering altogether we get
		\begin{align}
			\label{eq1.1}\tag*{$I_1 + I_1^\ast$}
			\hspace{-0.5cm}\langle [S_\alpha, A_\alpha]f ,f \rangle=& \frac{18 \alpha}{R^2} \norm{\partial_x^2 f}_{L^2}^2 + \frac{18 \alpha}{R^2} \norm{\partial_y^2 f}_{L^2}^2\\
			\label{eq2.1}\tag*{$I_2 +I_2^\ast$}
																							 & -\frac{15 \alpha^3}{R^4} \int (\partial_x f)^2 -\frac{15 \alpha^3}{R^4} \int (\partial_y f)^2\\ 
			 \label{eq3.1}\tag*{}
																							 & + \frac{144 \alpha^3}{R^4} \int \Big( \frac{x}{R} + \varphi(t) \Big)^2 (\partial_x f)^2 + \frac{144 \alpha^3}{R^4} \int \Big( \frac{y}{R} + \varphi(t) \Big)^2 (\partial_y f)^2\\
			\label{eq4.1}\tag*{}
																							 & - \frac{24 \alpha^3}{R^6} \int f^2 + 2\alpha \int (\varphi'(t) f)^2 + 2\alpha \int  \Big( \frac{x}{R} + \varphi(t) \Big) \varphi''(t) f^2\\
	    \label{eq5.1}\tag*{}
																							 &- \frac{24 \alpha^3}{R^6} \int f^2 + 2\alpha \int (\varphi'(t) f)^2  + 2\alpha \int  \Big( \frac{y}{R} + \varphi(t) \Big) \varphi''(t) f^2\\
			\label{eq6.1}\tag*{}
																							 & + \frac{48 \alpha^3}{R^3} \int \Big( \frac{x}{R} + \varphi(t) \Big)^2 \varphi'(t) f^2 + \frac{48 \alpha^3}{R^3} \int \Big( \frac{y}{R} + \varphi(t) \Big)^2 \varphi'(t) f^2  \\
			\label{eq7.1}\tag*{$I_3+ I_3^\ast$}
																							 & + \frac{288 \alpha^5}{R^6} \int \Big(\frac{x}{R} + \varphi(t) \Big)^4 f^2 + \frac{288 \alpha^5}{R^6} \int \Big(\frac{y}{R} + \varphi(t) \Big)^4 f^2\\
			\label{eq8.1}\tag*{}
																							 & + \frac{\alpha^3}{R^4} \norm{\partial_x f - \alpha \partial_x \phi f}_{L^2}^2 + \frac{\alpha^3}{R^4} \norm{\partial_y f - \alpha \partial_y \phi f}_{L^2}^2\\
		  \label{eq9.1}\tag*{}
																							 & -\frac{6 \alpha^5}{R^6} \int \abs*{\frac{\varkappa}{R} + \varphi(t) \xi}^4 f^2.
		\end{align}
	We consider~\eqref{eq2.1}, using again that $\abs*{\frac{\varkappa}{R} + \varphi(t) \xi}\geq 1,$ we obtain
		\begin{equation*}
			\begin{split}
				I_2 + I_2^\ast\geq& - \frac{15 \alpha^3}{R^4} \int \Big( \frac{x}{R} + \varphi(t) \Big)^2 (\partial_x f)^2 \underbrace{-\frac{15 \alpha^3}{R^4} \int \Big( \frac{y}{R} + \varphi(t) \Big)^2 (\partial_x f)^2}_{{}=:I}\\
													& \underbrace{- \frac{15 \alpha^3}{R^4} \int \Big( \frac{x}{R} + \varphi(t) \Big)^2 (\partial_y f)^2}_{{}=: II} - \frac{15 \alpha^3}{R^4} \int \Big( \frac{y}{R} + \varphi(t) \Big)^2 (\partial_y f)^2.
			\end{split}
		\end{equation*}
	First let us observe that, making use of integration by parts, $I$ can be re-written as
	\begin{equation}\label{Iterm}
		I= + \frac{15 \alpha^3}{R^4} \int \Big( \frac{y}{R} + \varphi(t) \Big)^2 f \partial_x^2 f.
	\end{equation}
	Using~\eqref{Iterm}, observing that $18=\frac{47}{4} + \frac{25}{4}$ and that $288=9 + 279$ we have
	\begin{equation*}
		\begin{split}
			I_1+I+ I_3^\ast
			 &=\int\Big( \frac{5}{2} \frac{\alpha^\frac{1}{2}}{R} \partial_x^2 f \Big)^2 + 2 \int \frac{15}{2} \frac{\alpha^3}{R^4} \Big( \frac{y}{R} + \varphi(t) \Big)^2 f \partial_x^2 f 
			+ \int \Big[ 3 \frac{\alpha^\frac{5}{2}}{R^3} \Big(\frac{y}{R} + \varphi(t) \Big)^2 f \Big]^2\\
											&\quad + \frac{47}{4}\frac{\alpha}{R^2}\int (\partial_x^2 f)^2 + 279\frac{\alpha^5}{R^6} \int \Big( \frac{y}{R} + \varphi(t) \Big)^4 f^2\\
											&=\int \Big[\frac{5}{2} \frac{\alpha^\frac{1}{2}}{R} \partial_x^2 f + 3 \frac{\alpha^\frac{5}{2}}{R^3} \Big( \frac{y}{R} + \varphi(t) \Big)^2 f \Big]^2 + \frac{47}{4}\frac{\alpha}{R^2}\int (\partial_x^2 f)^2 \\
											&\quad + 279\frac{\alpha^5}{R^6} \int \Big( \frac{y}{R} + \varphi(t) \Big)^4 f^2. 
		\end{split}
	\end{equation*}
	Proceeding in the same way for $I_1^\ast + II + I_3$ we get
	\begin{equation*}
		I_1^\ast + II + I_3\geq \int \Big[\frac{5}{2} \frac{\alpha^\frac{1}{2}}{R} \partial_y^2 f + 3 \frac{\alpha^\frac{5}{2}}{R^3} \Big( \frac{x}{R} + \varphi(t) \Big)^2 f \Big]^2 + \frac{47}{4}\frac{\alpha}{R^2}\int (\partial_y^2 f)^2 + 279\frac{\alpha^5}{R^6} \int \Big( \frac{x}{R} + \varphi(t) \Big)^4 f^2.  
	\end{equation*}
	Summing up, neglecting the two squares of binomial, that clearly are non negative, one has
	\begin{align}
			\label{eq1.2}\tag*{}
			\langle [S_\alpha, A_\alpha]f ,f \rangle=& \frac{47}{4}\frac{\alpha}{R^2} \norm{\partial_x^2 f}_{L^2}^2 + \frac{47}{4}\frac{\alpha}{R^2} \norm{\partial_y^2 f}_{L^2}^2\\
			\label{eq2.2}\tag*{}
																							 & +\frac{129 \alpha^3}{R^4} \int \Big(\frac{x}{R} + \varphi(t) \Big)^2 (\partial_x f)^2 +\frac{129 \alpha^3}{R^4} \int \Big(\frac{y}{R} + \varphi(t) \Big)^2  (\partial_y f)^2\\ 
			\label{eq3.2}\tag*{$I_1 + I_2 + I_3$}
																							 & - \frac{24 \alpha^3}{R^6} \int f^2 + 2\alpha \int (\varphi'(t) f)^2 + 2\alpha \int  \Big( \frac{x}{R} + \varphi(t) \Big) \varphi''(t) f^2\\
	    \label{eq4.2}\tag*{$I_1^\ast + I_2^\ast + I_3^\ast$}
																							 &- \frac{24 \alpha^3}{R^6} \int f^2 + 2\alpha \int (\varphi'(t) f)^2  + 2\alpha \int  \Big( \frac{y}{R} + \varphi(t) \Big) \varphi''(t) f^2\\
			\label{eq5.2}\tag*{$I_4+ I_5$}
																							 & + \frac{48 \alpha^3}{R^3} \int \Big( \frac{x}{R} + \varphi(t) \Big)^2 \varphi'(t) f^2 + \frac{279 \alpha^5}{R^6} \int \Big(\frac{x}{R} + \varphi(t) \Big)^4 f^2 \\
			\label{eq6.2}\tag*{$I_4^\ast+ I_5^\ast$}
																							 & + \frac{48 \alpha^3}{R^3} \int \Big( \frac{y}{R} + \varphi(t) \Big)^2 \varphi'(t) f^2  + \frac{279 \alpha^5}{R^6} \int \Big(\frac{y}{R} + \varphi(t) \Big)^4 f^2\\
			\label{eq7.2}\tag*{}
																							 & + \frac{\alpha^3}{R^4} \norm{\partial_x f - \alpha \partial_x \phi f}_{L^2}^2 + \frac{\alpha^3}{R^4} \norm{\partial_y f - \alpha \partial_y \phi f}_{L^2}^2\\
		  \label{eq8.2}\tag*{}
																							 & -\frac{6 \alpha^5}{R^6} \int \abs*{\frac{\varkappa}{R} + \varphi(t) \xi}^4 f^2.
		\end{align}
		Now we compute $I_2+ I_2^\ast + I_4+ I_4^\ast + I_5+ I_5^\ast,$ using that $(a^2 + b^2)\geq \frac{1}{2}(a+b)^2$ for all $a,b>0$ and that $\frac{279}{2}= 144 - \frac{9}{2}$ we have
		\begin{equation*}
			\begin{split}
			I_2+ I_2^\ast + I_4+ I_4^\ast + I_5+ I_5^\ast&\geq 4\alpha \int (\varphi'(t)f)^2 + \frac{48\alpha^3}{R^3} \int \abs*{\frac{\varkappa}{R} + \varphi(t) \xi}^2 \varphi'(t) f^2 + \frac{279}{2} \frac{\alpha^5}{R^6} \int \abs*{\frac{\varkappa}{R} + \varphi(t)\xi}^4 f^2\\
																									 & =\int \Big( 2 \alpha^\frac{1}{2} \varphi'(t) + 12 \frac{\alpha^\frac{5}{2}}{R^3} \abs*{\frac{\varkappa}{R} + \varphi(t)\xi}^2 \Big)^2 f^2 - \frac{9}{2} \frac{\alpha^5}{R^6} \int \abs*{\frac{\varkappa}{R} + \varphi(t)\xi}^4 f^2.
			\end{split}
		\end{equation*}
		Since we are assuming $\alpha^2\geq \norm{\varphi'}_\infty R^3$ and since $\abs*{\frac{\varkappa}{R} + \varphi(t)\xi}\geq 1,$ therefore
		\begin{equation*}
			2\alpha^\frac{1}{2} \varphi'(t)\geq -\frac{2 \alpha^\frac{5}{2}}{R^3} \geq -\frac{2 \alpha^\frac{5}{2}}{R^3} \abs*{\frac{\varkappa}{R} + \varphi(t)\xi}^2.
		\end{equation*}
		This gives
		\begin{equation*}
			I_2+ I_2^\ast + I_4+ I_4^\ast + I_5+ I_5^\ast\geq \frac{100\alpha^5}{R^6} \int \abs*{\frac{\varkappa}{R} + \varphi(t)\xi}^4 f^2 - \frac{9}{2} \frac{\alpha^5}{R^6} \int \abs*{\frac{\varkappa}{R} + \varphi(t) \xi}^4 f^2.
		\end{equation*}
		With regards to $I_3+I_3^\ast,$ assuming
		\begin{equation*}
			\alpha^2 \geq \norm{\varphi''}_\infty^\frac{1}{2} R^3,
		\end{equation*}
		and recalling that $\abs{\frac{\varkappa}{R}+ \varphi(t)\xi}\geq 1,$ we have
		\begin{equation*}
			\begin{split}
			I_3+ I_3^\ast&\geq -2\alpha \int \Big[\abs*{\frac{x}{R} + \varphi(t)} + \abs*{\frac{y}{R} + \varphi(t)}\Big] \norm{\varphi''}_\infty f^2\geq -\frac{2\sqrt{2} \alpha^5}{R^6} \int \abs*{\frac{\varkappa}{R} + \varphi(t)\xi} f^2\\
			&\geq - \frac{2\sqrt{2} \alpha^5}{R^6} \int \abs*{\frac{\varkappa}{R} + \varphi(t)\xi}^4 f^2. 
			\end{split}
		\end{equation*}
		Moreover
		\begin{equation*}
			I_1+ I_1^\ast= -\frac{48 \alpha^3}{R^6} \int f^2\geq -\frac{48 \alpha^3}{R^6} \int \abs*{\frac{\varkappa}{R} + \varphi(t)\xi}^4  f^2. 
		\end{equation*}
		Putting everything together and neglecting positive terms, we obtain the following estimate for the quantity $\langle [S_\alpha, A_\alpha]f, f\rangle:$
		\begin{equation*}
			\begin{split}
				\langle [S_\alpha, A_\alpha]f, f\rangle &\geq \frac{47}{4} \frac{\alpha}{R^2} \int (\partial_x^2 f)^2 + \frac{47}{4} \int (\partial_y^2 f)^2\\
&\phantom{~\geq} + \frac{129 \alpha^3}{R^4} \int \Big( \frac{x}{R} + \varphi(t) \Big)^2(\partial_x f)^2 + \frac{129 \alpha^3}{R^4} \int \Big( \frac{y}{R} + \varphi(t) \Big)^2(\partial_y f)^2\\
&\phantom{~\geq} +\big(100-\frac{9}{2} -2\sqrt{2} -48-6\big) \int \abs*{\frac{\varkappa}{R} + \varphi(t)\xi}^4 f^2\\
&\phantom{~\geq} + \frac{\alpha^3}{R^4} \norm{\partial_x f - \alpha \partial_x \phi f}_{L^2}^2 + \frac{\alpha^3}{R^4} \norm{\partial_y f - \alpha \partial_y \phi f}_{L^2}^2\\
&\geq \big( \frac{83}{2} -2\sqrt{2} \big) \frac{\alpha^5}{R^6} \int \abs*{\frac{\varkappa}{R} + \varphi(t)\xi}^4 f^2\\
&\phantom{~\geq} + \frac{\alpha^3}{R^4} \norm{\partial_x f - \alpha \partial_x \phi f}_{L^2}^2 + \frac{\alpha^3}{R^4} \norm{\partial_y f - \alpha \partial_y \phi f}_{L^2}^2. 
			\end{split}
		\end{equation*}
		Gathering the above information we conclude that
		\begin{multline*}
			\norm{e^{\alpha \phi(x,y,t)}(\partial_t + \partial_x^3 + \partial_y^3)e^{-\alpha \phi(x,y,t)} f}_{L^2}^2\\
			\geq \frac{\alpha^5}{R^6} \int \abs*{\frac{\varkappa}{R} + \varphi(t)\xi}^4 f^2 + \frac{\alpha^3}{R^4} \norm{\partial_x f - \alpha \partial_x \phi f}_{L^2}^2 + \frac{\alpha^3}{R^4} \norm{\partial_y f - \alpha \partial_y \phi f}_{L^2}^2
		\end{multline*}
		holds. Then a straightforward computation shows that this easily gives~\eqref{Carleman} in terms of $g$ with $c=\sqrt{3}.$
	\end{proof}

\medskip
Next, we shall extend the result in Lemma~\ref{Carleman1} to operators of the form~\eqref{operator_difference}. 

	\begin{lemma}\label{Carleman2}
	Assume that $\varphi\colon [0,1] \to \R$ is a smooth function. Then there exist $c>0,$ $R_0= R_0(\norm{\varphi'}_{\infty},\norm{\varphi''}_{\infty},$ $\norm{a_0}_{\infty},\norm{a_1}_{\infty})>1$ and $M_1=M_1(\norm{\varphi'}_\infty,\norm{\varphi''}_{\infty})>0$ such that the inequality 
	\begin{multline}\label{Carleman_bis}
		\frac{\alpha^\frac{5}{2}}{R^3} \norm*{e^{\alpha \phi(x,y,t)} \phi(x,y,t)\, g}_\Lt + \frac{\alpha^\frac{3}{2}}{R^2} \norm{ e^{\alpha \phi(x,y,t)} \abs{\nabla g}}_\Lt\\
		\leq c \norm{ e^{\alpha \phi(x,y,t)}(\partial_t + \partial_x^3 + \partial_y^3 + a_1(x,y,t)(\partial_x + \partial_y) + a_0(x,y,t))g}_\Lt
	\end{multline}
	holds for $R\geq R_0,$ $\alpha$ such that $\alpha^2\geq M_1 R^3,\, g \in C^\infty_0(\R^2 \times [0,1])$ supported in
	\begin{equation*}
		\Big\{ (x,y,t) \in \R^2 \times[0,1]\colon \abs*{\frac{\varkappa}{R} + \varphi(t) \xi}\geq 1  \Big\}
	\end{equation*}
	and $\phi(x,y,t)= \abs*{\frac{\varkappa}{R} + \varphi(t) \xi}^2=\big( \frac{x}{R} + \varphi(t) \big)^2 + \big( \frac{y}{R} + \varphi(t) \big)^2,$ with $\varkappa=(x,y)$ and $\xi=(1,1).$ 
	\end{lemma}
	\begin{proof}
		
		From the estimate~\eqref{Carleman} of Lemma~\ref{Carleman1}, adding and subtracting the lower order terms, it follows that
		\begin{equation}\label{absorption}
			\begin{split}
				\frac{\alpha^\frac{5}{2}}{R^3} \norm{e^{\alpha \phi} \phi g}_{L^2} + \frac{\alpha^\frac{3}{2}}{R^2} \norm{e^{\alpha \phi} \abs{\nabla g}}_{L^2}
				&\leq c \norm{e^{\alpha \phi}(\partial_t + \partial_x^3 + \partial_y^3)g}_{L^2}\\
														&\leq c \norm{e^{\alpha \phi} (\partial_t + \partial_x^3 + \partial_y^3 + a_1 (\partial_x + \partial_y) + a_0)g}_{L^2} + c \norm{e^{\alpha \phi}(a_1(\partial_x + \partial_y) + a_0)g}_{L^2}\\
														&\leq c \norm{e^{\alpha \phi} (\partial_t + \partial_x^3 + \partial_y^3 + a_1 (\partial_x + \partial_y) + a_0)g}_{L^2} + \sqrt{2}c \norm{e^{\alpha \phi} \abs{\nabla g}}_{L^2} \norm{a_1}_{L^\infty}\\
														&\quad + c \norm{e^{\alpha\phi} \phi g}_{L^2} \norm{a_0}_{L^\infty},
			\end{split}
		\end{equation}
		where the last inequality follows from the assumption $\abs*{\frac{\varkappa}{R} + \varphi(t)\xi}\geq 1.$
		
		Under our hypothesis $\alpha^2 \geq M_1 R^{\,3},$ the ratios $\frac{\alpha^\frac{3}{2}}{R^2}$ and then $\frac{\alpha^\frac{5}{2}}{R^2}$ on the left-hand side grow as a positive fractional power of $R,$ therefore, being the last two terms on the right-hand side finite because of the strong assumption about $g,$ these can be absorbed on the left-hand side assuming $R$ to be sufficiently large. 
		
		This yields the desired result. 
			\end{proof}

		\begin{remark}
			Notice that our hypothesis $\alpha^2\geq M_1 R^{\,3}$ turns out to be fundamental to make the term $\frac{\alpha^\frac{3}{2}}{R^2}$ growing as a positive fractional power of $R$ in order to absorb $\norm{e^{\alpha\phi}\abs{\nabla g}}_{L^2}$ in the left-hand side of~\eqref{absorption} and obtain~\eqref{Carleman_bis}. We recall that the term $\norm{e^{\alpha\phi}\abs{\nabla g}}_{L^2}$ comes from the fact that we want to obtain a Carleman estimate for the operator $P=\partial_t + \partial_x^3 + \partial_y^3 + a_1(\partial_x +\partial_y) + a_0$ which involves first order derivatives. If instead we consider an operator of this form $\partial_t + \partial_x^3 + \partial_y^3 + a(x,y,t),$	namely an operator in which the first derivatives do not appear, it would be sufficient to assume $\alpha^4 \geq M_1 R^{\,5}$ in order to guarantee that $\frac{\alpha^\frac{5}{2}}{R^3}$ grows as a fractional power of $R$ and hence to get from estimate~\eqref{Carleman} a Carleman estimate for this operator. In particular, this means that the form of the operator plays the fundamental role in the choice of the decay necessary in order to obtain a unique continuation result. To be more precise, considering the differential equation
			\begin{equation*}
				\partial_t u + \partial_x^3 u +\partial_y^3 u + a(x,y,t)u=0,
			\end{equation*}
			in this case it would be sufficient to require weaker hypothesis about the decay of the solution than the ones in Theorem~\ref{main_result_s}, namely $u_1(0)-u_2(0);\, u_1(1)-u_2(1) \in L^2(e^{a\abs{x + y}^\frac{5}{4}}\, dx\, dy).$ 
		
		\medskip
		Several evidences of the strict link between the decay assumption necessary to get unique continuation results and the form of the operator one is dealing with can be found in literature. In~\cite{Dawson}, Liana Dawson proved the following result concerning with unique continuation for equations in the KdV hierarchy. 
		\begin{theorem}\label{Dawson_result}
	Let $u_1, u_2$ two sufficiently smooth solutions of
	\begin{equation*}
		\partial_t u + \partial_x^5 u + 10u \partial_x^3 u + 20 \partial_x u \partial_x^2 u + 30 u^2 \partial_x u=0,\qquad (x,t)\in \R\times [0,1]. 
	\end{equation*}
	If there exists an $\varepsilon>0$ such that
	\begin{equation*}
		u_1(0)-u_2(0), u_1(1)-u_2(1)\in H^2(e^{a x_+^{4/3 + \varepsilon}} \, dx)
	\end{equation*}
	for $a>0$ sufficiently large,  then $u_1\equiv u_2.$
\end{theorem} 
		
The previous result comes out as a consequence of the analogous linear result for the equation with variable coefficients
\begin{equation*} 
	\partial_t v + \partial_x^5 v + a_4(x,t)\partial_x^4 v + a_3(x,t)\partial_x^3 v + a_2(x,t)\partial_x^2 v + a_1(x,t)\partial_x v + a_0(x,t) v=0,
\end{equation*}
or better, since it is always possible to eliminate the fourth order term by considering $w(x,t):= u(x,y) e^{\frac{1}{5} \int_0^x a_4(s,t)\, ds},$ for the equation
\begin{equation}\label{eq_daw2} 
	\partial_t v + \partial_x^5 v + a_3(x,t)\partial_x^3 v + a_2(x,t)\partial_x^2 v + a_1(x,t)\partial_x v + a_0(x,t) v=0.
\end{equation} 
In order to get the result, a Carleman estimate for the leading order terms of the operator, namely $\partial_t + \partial_x^5,$ was shown:
\begin{multline*}
	\frac{\alpha^\frac{1}{2}}{R} \norm*{e^{\alpha \big(\tfrac{x}{R} + \varphi(t)\big)^2} \partial_x^4 g}_{L^2} + \frac{\alpha^\frac{3}{2}}{R^2} \norm*{e^{\alpha \big(\tfrac{x}{R} + \varphi(t)\big)^2} \Big( \frac{x}{R} + \varphi(t)\Big)\partial_x^3 g}_{L^2} \\
	+ \frac{\alpha^\frac{5}{2}}{R^3} \norm*{e^{\alpha \big(\tfrac{x}{R} + \varphi(t)\big)^2} \Big( \frac{x}{R} + \varphi(t)\Big)^2\partial_x^2 g}_{L^2}
	+\frac{\alpha^\frac{7}{2}}{R^4} \norm*{e^{\alpha \big(\tfrac{x}{R} + \varphi(t)\big)^2} \Big( \frac{x}{R} + \varphi(t)\Big)^3\partial_x g}_{L^2}\\
	+ \frac{\alpha^\frac{9}{2}}{R^5} \norm*{e^{\alpha \big(\tfrac{x}{R} + \varphi(t)\big)^2} \Big( \frac{x}{R} + \varphi(t)\Big)^4 g}_{L^2}
	\leq c \norm{e^{\alpha \big(\tfrac{x}{R} + \varphi(t)\big)^2} (\partial_t + \partial_x^5) g}_{L^2}.
\end{multline*} 
As in our case, in order to obtain from this a Carleman estimate for the operator involving the lower order derivatives, that is $\partial_t + \partial_x^5 + a_3 \partial_x^3 + a_2 \partial_x^2 + a_1 \partial_x + a_0,$ an ``adding and subtracting argument'' is performed. To let this argument work we need to choose $\alpha$ in such a way the ratios $\frac{\alpha^\frac{3}{2}}{R^2}, \frac{\alpha^\frac{5}{2}}{R^3}, \frac{\alpha^\frac{7}{2}}{R^4}, \frac{\alpha^\frac{9}{2}}{R^5}$ grow as fractional powers of $R$ because, therefore, for $R$ sufficiently large, the additional terms on the right-hand side can be absorbed in the left-hand side. This entails the restriction $\alpha^3\geq M_1 R^{4+\varepsilon}$ about $\alpha$ which leads to the exponential decay rate in Theorem~\ref{Dawson_result}.  

Let us observe that also in this fifth order setting, if one considered a differential equation in which \emph{third} and fourth derivatives do not appear, namely
\begin{equation}\label{eq_daw1}
	\partial_tv + \partial_x^5 v + a_2(x,t) \partial_x^2 v + a_1(x,t) \partial_x v + a_0(x,t) v=0,
\end{equation}
 in this case we just need to guarantee that $\frac{\alpha^\frac{5}{2}}{R^3}, \frac{\alpha^\frac{7}{2}}{R^4}, \frac{\alpha^\frac{9}{2}}{R^5}$ grow as a fractional positive power of $R,$ that holds true assuming $\alpha ^4\geq M_1 R^5.$ This means that in this situation a stronger unique continuation result could be achieved requiring a weaker decay rate for the solutions at two distinct times. 

In~\cite{Isaza}, it was proved that this fact holds for a quite general class of high order equations of KdV type, which includes the KdV hierarchy. 
Precisely that work is concerned with unique continuation results for the equation
\begin{equation}\label{Isaza_equation}
	\partial_t v +(-1)^{k+1} \partial_x^n v + P(v, \partial_x v, \dots, \partial_x^p v)=0, \qquad (x,t)\in \R\times [0,1],
\end{equation}  
where $n=2k+1,$ $k=1,2, \dots$ and $P$ is a polynomial in $v, \partial_x v, \dots, \partial_x^p v,$ with $p\leq n-1.$ Of particular interest in that work were the cases $p=n-2$ and $p\leq k$ with $n\geq 5.$ 
For these situations it was proved that if the difference of two sufficiently smooth solutions of the equation~\eqref{Isaza_equation} with $p=n-2$ decays as $e^{- x_+^{4/3 + \varepsilon}}$ at two distinct times, then $u_1\equiv u_2.$ Moreover when $p\leq k$ a similar result was obtained assuming the weaker decay $e^{-a x_+^{n/n-1}}$ for $a>0$ sufficiently large.
\end{remark}	
		
\medskip
Now we are in position to prove the lower bound.
\begin{proof}[Proof of Theorem~\eqref{lower_bound}]
The starting point in the proof of the lower bound is to apply estimate~\eqref{Carleman_bis} in Lemma~\ref{Carleman2} to a particular function $g$ that we shall define to be suitably related with the solution $v$ of~\eqref{generalized_Z-K_diff}, in order to do that for $R>2$ we introduce the function $\theta_R\in C^\infty(\R^2)$ so defined
\begin{equation*}
\theta_R(x,y)=
	\begin{system}
		1 \qquad \text{if}\quad \abs{x+y}<R-1\quad \land\quad \abs{x-y}<R-1\\
		0 \qquad \text{if}\quad \abs{x+y}>R \phantom{-1~}\,\quad \land\quad \abs{x-y}>R
	\end{system},
\end{equation*}
$\mu\in C^\infty(\R^2)$ such that
\begin{equation*}
\mu(x,y)=
	\begin{system}
		0 \qquad \text{if}\quad \sqrt{x^2+y^2}<1\vspace{0.2cm}\\
		1 \qquad \text{if}\quad \sqrt{x^2+y^2}>2
	\end{system}
\end{equation*}
and $\varphi\colon \R \to [0, 2\sqrt{2}],$ $\varphi \in C_0^\infty(\R)$ with
		\begin{equation*}
			\varphi(t)=
			\begin{system}
				0 \qquad \hspace{0.45cm} \text{if}\quad t \in \big[ 0, \frac{r}{2}\big] \cup \big[ 1- \frac{r}{2},1 \big]\\
				2\sqrt{2} \qquad \text{if}\quad t \in [r, 1-r] 
			\end{system},
		\end{equation*}
		increasing in $[\frac{r}{2}, r]$ and decreasing in $[1-r, 1- \frac{r}{2}].$
		
		We define the auxiliary function
		\begin{equation*}
			g(x,y,t)= \theta_R(x,y)\,\mu\Big(\frac{x}{R} + \varphi(t), \frac{y}{R} + \varphi(t)\Big)\, v(x,y,t), \qquad (x,y) \in \R^2,\, t\in[0,1].
		\end{equation*}
		It is easy to see that $g$ satisfies
		\begin{equation*}
			\begin{split}
				(&\partial_t + \partial_x^3 + \partial_y^3 + a_1 (\partial_x + \partial_y) + a_0) g\\
											&=\mu\Big(\frac{x}{R} + \varphi(t), \frac{y}{R} + \varphi(t)\Big) \Big[3\partial_x \theta_R \partial_x^2 v + 3 \partial_y \theta_R \partial_y^2 v + 3 \partial_x^2 \theta_R \partial_x v + 3 \partial_y^2 \theta_R \partial_y v + \partial_x^3 \theta_R v + \partial_y^3 \theta_R v\\
											& \hspace{+4.2cm} + a_1 \partial_x \theta_R v + a_1 \partial_y \theta_R v \Big]\\
				& \quad +3 R^{-1} \theta_R \partial_x \mu\, \partial_x^2 v + 3 R^{-1} \theta_R \partial_y \mu\, \partial_y^2 v + 3 R^{-1} \big[( R^{-1} \theta_R \partial_x^2 \mu + 2 \partial_x \theta_R \partial_x \mu) \partial_x v + (R^{-1} \theta_R \partial_y^2 \mu + 2 \partial_y \theta_R \partial_y \mu)\partial_y v\big]\\
				& \quad + \Big[ \theta_R \partial_x \mu \Big(\varphi' + \frac{a_1}{R} \Big) + \theta_R \partial_y \mu \Big(\varphi' + \frac{a_1}{R} \Big) + R^{-3} \partial_x^3 \mu + R^{-3} \partial_y^3 \mu + 3 R^{-1} \partial_x^2 \theta_R \partial_x \mu + 3 R^{-1} \partial_y^2 \theta_R \partial_y \mu\\
				& \qquad \, \, + 3 R^{-2} \partial_x \theta_R \partial_x^2 \mu + 3 R^{-2} \partial_y \theta_R \partial_y^2 \mu  
								\Big] v.
			\end{split}
		\end{equation*}
		\begin{remark}\label{eq_g}
			Observe that since in the first term in the right-hand side of the previous equation the derivatives of $\theta_R$ appear, this term is supported in $\{(x,y) \colon R-1 \leq \abs{x+y} \leq R, \, R-1 \leq \abs{x-y} \leq R\}\times [0,1],$ so in particular $(x,y)$ is such that $\sqrt{x^2+y^2}\leq R$ and this entails $\abs*{\frac{\varkappa}{R} + \varphi(t)\xi}\leq 5.$ Moreover, one can notice that all the remaining terms, sorted with respect to their dependence on the derivatives of our solution $v,$  contain the derivatives of $\mu,$ this means that they are supported in $\{ (x,y,t)\colon 1 \leq\abs*{\frac{\varkappa}{R} + \varphi(t)\xi}\leq 2, \, t \in [0,1]\}.$
		 \end{remark}
		
		Next we verify that function $g$ defined above satisfies the hypotheses of Lemma~\ref{Carleman2}. Indeed 
		\begin{itemize}
			\item if $\abs{x+y}>R \, \land \abs{x-y}>R$  then we fall outside the support of $\theta_R,$ this means that $g(x,y,t)=0.$
			\item if $\abs{x+y}<R \, \land \abs{x-y}<R$ and $t \in [0, \frac{r}{2}] \cup [1- \frac{r}{2}, 1]$ then $g(x,y,t)=0,$ indeed being $\abs{x+y}<R\, \land \abs{x-y}<R$ in particular $\sqrt{x^2+y^2}<R$ and since $\varphi(t)=0$ if $t \in [0, \frac{r}{2}] \cup [1- \frac{r}{2}, 1],$ this gives $\abs*{\frac{\varkappa}{R} + \varphi(t)\xi}<1,$ therefore we are out of the support of $\mu\big(\frac{x}{R} + \varphi(t), \frac{y}{R} + \varphi(t) \big)$ and so  $g(x,y,t)=0.$ 
		\end{itemize}
		This guarantees that $g$ is compactly supported.
		
		Now we observe that $g$ is supported in $\{ (x,y,t)\in \R^2 \times [0,1] \colon \abs*{\frac{\varkappa}{R} + \varphi(t)\xi}\geq 1\},$ indeed if $\abs*{\frac{\varkappa}{R} + \varphi(t)\xi}<1$ then $\mu\big( \frac{x}{R} + \varphi(t), \frac{y}{R} + \varphi(t) \big)=0$ and so $g(x,y,t)=0.$
		
		\smallskip
		Since $g$ satisfies the hypotheses of Lemma~\ref{Carleman2}, there exist $c>0, R_0$ and $M_1$ such that 
		\begin{equation}\label{first_in}
			c \frac{\alpha^\frac{5}{2}}{R^3} \norm{ e^{\alpha \phi} g}_{L^2(\R^2\times[0,1])} \leq \norm{e^{\alpha \phi}(\partial_t + \partial_x^3 + \partial_y^3 + a_1 (\partial_x + \partial_y) + a_0)g}_{L^2(\R^2\times[0,1])}.
		\end{equation}
		Recalling that $\phi(x,y,t)=\abs*{\frac{\varkappa}{R} + \varphi(t)\xi}^2$ and making use of Remark~\ref{eq_g} it is easy to see that
		\begin{equation}\label{second_in}
			\norm{e^{\alpha \phi}(\partial_t + \partial_x^3 + \partial_y^3 + a_1 (\partial_x + \partial_y) + a_0)g}_{L^2(\R^2\times[0,1])} \leq c_1\, e^{25 \alpha} A_R(v) + c_2\, e^{4\alpha} A.
		\end{equation}
		
		\smallskip
		We observe that in the region $Q=\{ (x,y,t)\colon \sqrt{x^2+ y^2}\leq 1, \, t\in [r, 1-r]\}$ we have $g(x,y,t)=v(x,y,t).$ Indeed if $\sqrt{x^2+y^2}\leq 1,$ in particular $\sqrt{x^2+ y^2}\leq R-1,$ therefore $\theta_R(x,y)\equiv 1.$ 
		
		Moreover in $Q$ it also holds that $\mu\big(\frac{x}{R} + \varphi(t), \frac{y}{R} + \varphi(t) \big)\equiv 1.$ Indeed, using that $\varphi(t)\equiv 2\sqrt{2}$ in $[r,1-r],$ the trivial inequality $a+b\leq \sqrt{2} \, \sqrt{a^2 +b^2}$ which holds for all $a,b>0$ and the assumption $R>2,$ we have
		\begin{equation*}
			\abs*{\frac{\varkappa}{R} + 2\sqrt{2}\xi}^2
			=16 + \frac{x^2+y^2}{R^2} + \frac{4\sqrt{2}}{R}(x+y)\geq 16 - \frac{4\sqrt{2}}{R}(\abs{x}+ \abs{y}) 
			\geq 16 -\frac{8}{R}\sqrt{x^2+y^2}\geq 12.
		\end{equation*}
			Then $\abs*{\frac{\varkappa}{R} + \varphi(t)\xi}>\sqrt{12}>2$ and so $\mu\big(\frac{x}{R} + \varphi(t), \frac{y}{R} + \varphi(t) \big)\equiv 1.$
			
		Using that $g\equiv v$ in $Q$ we obtain the following chain of inequalities:
		\begin{equation}\label{third_in}
			c \frac{\alpha^{\frac{5}{2}}}{R^3} \norm{e^{\alpha \phi} g}_{L^2(\R^2\times[0,1])}\geq c \frac{\alpha^{\frac{5}{2}}}{R^3} \norm{e^{\alpha \phi} g}_{L^2(Q)}=c \frac{\alpha^{\frac{5}{2}}}{R^3} \norm{e^{\alpha \phi} v}_{L^2(Q)}\geq c \frac{\alpha^\frac{5}{2}}{R^3} e^{4\alpha} \norm{v}_{L^2(Q)}. 
		\end{equation}
		From~\eqref{first_in},~\eqref{second_in},~\eqref{third_in} and the assumption $\norm{v}_{L^2(Q)}>\delta$ we get 
		\begin{equation*}
			c \frac{\alpha^\frac{5}{2}}{R^3} e^{4\alpha} \delta \leq c_1\, e^{25 \alpha} A_R(v) + c_2\, e^{4\alpha} A, 
		\end{equation*}
		therefore
		\begin{equation*}
			c \frac{\alpha^\frac{5}{2}}{R^3} \delta \leq c_1 \, e^{21 \alpha} A_R(v) + c_2 A.
		\end{equation*}
		Taking $\alpha=M_1^\frac{1}{2} R^\frac{3}{2}$ with $M_1$ as in Lemma~\ref{Carleman2} we obtain
		\begin{equation*}
			c M_1^\frac{5}{4} R^\frac{3}{4} \delta \leq c_1 e^{21 M_1^\frac{1}{2} R^\frac{3}{2}} A_R(v) + c_2 A.
		\end{equation*}
		Now if we take $R$ large enough, the second term on the right-hand side of the previous inequality can be absorbed by the term on the left-hand side, so we can conclude that there exists $\widetilde{R}_0>0$ such that for $R\geq \widetilde{R}_0$ the following holds 
		\begin{equation*}
			A_R(v)\geq \frac{c}{2} e^{-21 M_1^\frac{1}{2} R^\frac{3}{2}}.
		\end{equation*}
		This yields the desired result.
\end{proof}


\subsection{Upper bound}
Now we will turn on the proof of the upper bound for solutions of~\eqref{generalized_Z-K_diff}. 
Precisely the result we will prove is the following.

\begin{theorem}\label{upper_bound}
	Assume that the coefficients in~\eqref{generalized_Z-K_diff} $a_0, a_1$ satisfy $a_0\in L^\infty \cap L_x^2 L_{y t}^\infty$ and $a_1 \in L_x^2 L_{y t}^\infty\cap L_x^1 L_{y t}^\infty.$  If $v\in C([0,1]; H^4(\R^2))$ is a solution of~\eqref{generalized_Z-K_diff} satisfying that 
	\begin{equation*}
		v(0), v(1) \in L^2(e^{a\abs{x+y}^\frac{3}{2}} dx\, dy)
	\end{equation*}
	for some $a>0$, then there exist $c$ and $R_0>0$ sufficiently large such that for $R\geq R_0$
	\begin{equation*}
		\norm{v}_{L^2(Q_R \times [0,1])} + \sum_{ 0< k+l\leq 2} \norm{\partial_x^k \partial_y^l v}_{L^2(Q_R \times [0,1])}\leq c e^{-a \left(\frac{R}{18}\right)^\frac{3}{2}},
	\end{equation*}
	where $ Q_R:=\{ (x, y) \colon R-1\leq \abs{x+y}\leq R\, \land \, R-1\leq \abs{x-y}\leq R \}.$
\end{theorem}

As in~\cite{B_I_M} we shall prove first the following lemma whose proof can be found in the Appendix.

\begin{lemma}\label{lemma2.2bis}
	Let $w \in C([0,1]; H^4(\R^2)) \cap C^1([0,1]; L^2(\R^2))$ such that for all $t \in [0,1]$ $\supp w(t)\subseteq K,$ where $K$ is a compact subset of $\R^2.$ 
	
	Assume that $a_0\in L^\infty \cap L_x^2 L_{y t}^\infty$ and $a_1 \in L_x^2 L_{y t}^\infty\cap L_x^1 L_{y t}^\infty,$ with small norms in these spaces. 
	
	Then there exists $c>0,$ independent of the set $K,$ such that for $\beta\geq 1$ the following estimate holds
	\begin{multline*}
		\norm{e^{\beta \abs{x+y}} w}_{L^2(\R^2\times[0,1])} + \sum_{0< k+ l \leq 2} \norm{e^{\beta \abs{x+y}} \partial_x^k \partial_y^l w}_{L_x^\infty\, L_{y t}^2(\R^2\times[0,1])} 
		\\ \leq  c\, \beta^2\big (\norm{J^3(e^{\beta \abs{x+y}} w(0))}_{L^2(\R^2)} + \norm{J^3(e^{\beta \abs{x+y}} w(1))}_{L^2(\R^2)} \big)
									 \\ + c\, \norm{e^{\beta \abs{x+y}} (\partial_t +\partial_x^3 + \partial_y^3 +a_1(\partial_x +\partial_y) +a_0)w}_{L_t^1 L_{x y}^2 \cap L_x^1 L_{y t}^2(\R^2\times[0,1])}
\end{multline*} 		
with $J$ such that $\widehat{J g}(\xi, \eta):= (1 + \xi^2 + \eta^2)^\frac{1}{2} \widehat{g}(\xi, \eta).$ (Here, $\widehat{\phantom{g}}$ denotes the spatial Fourier transform in $\R^2$ and $(\xi,\eta)$ are the variables in the frequency space corresponding to the space variables $(x,y).$)
\end{lemma}

\begin{remark}
	Although we have assumed $w(t)$ to be compactly supported for all $t\in [0,1],$ we shall see that the argument in Lemma~\ref{lemma2.2bis} can be extended to a larger class of functions. Indeed we only have to ensure that the following quantity
	\begin{equation*}
			\norm{e^{\beta \abs{x+y}} w}_{L^2} + \sum_{0< k+ l \leq 2} \norm{e^{\beta \abs{x+y}} \partial_x^k \partial_y^l w}_{L_x^\infty\, L_{y t}^2}
		\end{equation*}
		is finite.
		
		To justify this affirmation we need to anticipate few facts regarding the proof of Lemma~\ref{lemma2.2bis} that can be found in the Appendix below. Just to convey the idea we introduce the following notation similar to that one already used by Escauriaza, Kenig, Ponce and Vega in~\cite{E_K_P_V2}
		\begin{align*}
			&\normeq{h}_1:=\norm{e^{\beta\abs{x+y}} h}_{L^2(\R^2\times[0,1])} + \sum_{0<k+l\leq 2} \norm{e^{\beta \abs{x+y}} \partial_x^k \partial_y^l h}_{L^\infty_x L^2_{y t}(\R^2\times [0,1])},\\
			&\normeq{h}_2:=\norm{h}_{L^1_t L^2_{x y} \cap L^1_x L^2_{y t}(\R^2 \times [0,1])}
		\end{align*}
		
		As customary we will start proving the estimate involving just the leading part of the operator we are working with, namely $H:=\partial_t + \partial_x^3+ \partial_y^3,$ then we extended our \emph{a priori} estimate to the whole operator, that is $H_a:= \partial_t + \partial_x^3 + \partial_y^3 + a_1(\partial_x +\partial_y) +a_0,$ using an ``adding and subtracting'' argument. More precisely, using our new notation, the starting estimate for the leading operator $H$ can be written as
		\begin{equation*}
			\normeq{w}_1\leq c\,\beta^2 \big(\norm{J^3(e^{\beta \abs{x+y}} w(0))}_{L^2(\R^2)} + \norm{J^3(e^{\beta \abs{x+y}} w(1))}_{L^2(\R^2)} \big) + \normeq{e^{\beta \abs{x+y}}H w}_2.
		\end{equation*}
		
		Now the second step we follow is to extend this inequality to the operator $H_a.$ Using the H\"older inequality and the smallness assumptions about $a_0$ and $a_1$ we get
		\begin{equation*}
			\begin{split}
					\normeq{w}_1&\leq c\,\beta^2 \big(\norm{J^3(e^{\beta \abs{x+y}} w(0))}_{L^2(\R^2)} + \norm{J^3(e^{\beta \abs{x+y}} w(1))}_{L^2(\R^2)} \big) + \normeq{e^{\beta \abs{x+y}}H w}_2\\
											&\leq c\,\beta^2 \big(\norm{J^3(e^{\beta \abs{x+y}} w(0))}_{L^2(\R^2)} + \norm{J^3(e^{\beta \abs{x+y}} w(1))}_{L^2(\R^2)} \big)\\
											&\phantom{~\leq} + \normeq{e^{\beta \abs{x+y}}H_a w}_2 + \normeq{e^{\beta \abs{x+y}}\big(a_1(\partial_x +\partial_y) +a_0\big) w}_2\\
											&\leq c\,\beta^2 \big(\norm{J^3(e^{\beta \abs{x+y}} w(0))}_{L^2(\R^2)} + \norm{J^3(e^{\beta \abs{x+y}} w(1))}_{L^2(\R^2)} \big)\\
											&\phantom{~\leq} + \normeq{e^{\beta \abs{x+y}}H_a w}_2 + \norm{a_0}_{L^\infty \cap L^2_x L^\infty_{y t}} \norm{e^{\beta \abs{x+y}} w}_{L^2} + \norm{a_1}_{L^2_x L^\infty_{y t} \cap L^1_x L^\infty_{y t}} \norm{e^{\beta \abs{x+y}} (\partial_x + \partial_y)w}_{L^\infty_x L^2_{y t}}\\
											&\leq c\,\beta^2 \big(\norm{J^3(e^{\beta \abs{x+y}} w(0))}_{L^2(\R^2)} + \norm{J^3(e^{\beta \abs{x+y}} w(1))}_{L^2(\R^2)} \big)\\
											&\phantom{~\leq} + \normeq{e^{\beta \abs{x+y}}H_a w}_2 + \frac{1}{2} \normeq{w}_1.
				\end{split}
		\end{equation*}
	Hence, if we are working on a class of solutions $w=w(x,y,t)$ for which $\normeq{w}_1$ is finite for all $\beta>0,$ we can obtain the desired result, that is the extended a priori estimate
		\begin{equation*}
			\normeq{w}_1\leq c\,\beta^2 \big(\norm{J^3(e^{\beta \abs{x+y}} w(0))}_{L^2(\R^2)} + \norm{J^3(e^{\beta \abs{x+y}} w(1))}_{L^2(\R^2)} \big) + c \normeq{e^{\beta \abs{x+y}}H_a w}_2.
		\end{equation*}
		Now it should appear clear that in order to obtain the result in Lemma~\ref{lemma2.2bis} the hypothesis $w(t)$ to be compactly supported is overabundant, it is sufficient to produce solutions for which $\normeq{w}_1$ is finite. 
		
		In order to ensure the finiteness of $\normeq{w}_1$ it is sufficient to prove that a so-called persistence property for the solution flow to~\eqref{Z-K_s} holds true. More precisely it is sufficient to prove that if a solution to~\eqref{Z-K_s} has a suitable exponential decay at two different instant of time, namely $t_1=0$ and $t_2=1,$ then this decay rate is preserved for all $t\in [0,1].$ The proof of this property and thus of the finiteness of the $\normeq{w}_1$ can be found in Appendix~\ref{Appendix:persistence} (Theorem~\ref{persistence_properties}).
\end{remark}

Now we are in position to prove the upper estimate in Theorem~\ref{upper_bound}. 
	\begin{proof}[Proof of Theorem~\ref{upper_bound}]
		We construct a $C^\infty$ truncation function $\mu_R$ with $\mu_R(x,y)=0$ if $\abs{x+y}\leq R$ and $\mu_R(x,y)=1$ if $\abs{x+y}\geq \frac{18 R -1}{4}.$
		
		Let us define 
		\begin{equation*}
			w(x,y,t):= \mu_R(x,y) v(x,y,t).
		\end{equation*}
		Now we want to see what kind of equation is satisfied by $w.$ It is easy to see that, since $v$ is a solution of~\eqref{generalized_Z-K_diff}, the following holds
		\begin{equation*}
			\big(\partial_t + \partial_x^3 + \partial_y^3 + a_1(x,y,t)(\partial_x + \partial_y) + a_0(x,y,t)\big) w = e_R(x,y,t),
		\end{equation*}
		where 
		\begin{equation*}
			e_R(x,y,t)=\partial_x^3 \mu_R v + 3\partial_x^2\mu_R \partial_x v + 3 \partial_x \mu_R \partial_x^2 v + \partial_y^3 \mu_R v + 3\partial_y^2\mu_R \partial_y v + 3 \partial_y \mu_R \partial_y^2 v + a_1(x,y,t) \partial_x \mu_R v + a_1(x,y,t) \partial_y \mu_R v.
		\end{equation*}
		This means that our function $w$ solves an equation like~\eqref{generalized_Z-K_diff} but with a correction term $e_R.$
		Next step would be apply Lemma~\ref{lemma2.2bis} to function $w.$ To do so, we first need $a_0, a_1$ to have small norms. Therefore we introduce $\widetilde{\mu_R}$ such that $\widetilde{\mu_R} \mu_R(x,y)= \mu_R(x,y),$ and $\widetilde{a}_j:=a_j(x,y,t) \widetilde{\mu}_R$ with $j=0,1$ have small norms in the corresponding spaces for $R\geq R_0.$

Let us consider the operator
\begin{equation}\label{operator_tilde}
	\widetilde{L}:= \partial_t + \partial_x^3 + \partial_y^3 + \widetilde{a}_1 (\partial_x + \partial_y) + \widetilde{a}_0,
\end{equation}	

Now we are in position to apply~\eqref{lemma2.2bis} with the operator $\widetilde{L}.$ This gives
\begin{multline}\label{last_one}
		\norm{e^{\beta \abs{x+y}} w}_{L^2} + \sum_{0< k+ l \leq 2} \norm{e^{\beta \abs{x+y}} \partial_x^k \partial_y^l w}_{L_x^\infty\, L_{y t}^2}\\
			\leq  c\, \beta^2\big (\norm{J^3(e^{\beta \abs{x+y}} w(0))}_{L^2} + \norm{J^3(e^{\beta \abs{x+y}} w(1))}_{L^2} \big)
									 + c\, \norm{e^{\beta \abs{x+y}} e_R}_{L_t^1 L_{x y}^2 \cap L_x^1 L_{y t}^2}.
\end{multline} 		
\begin{remark}
	With an abuse of notation we have called $\widetilde{e}_R$ (the corresponding remainder coming from the action of $\widetilde{L}$ on $w$) as $e_R.$
\end{remark}

We consider the term $c\, \beta^2\norm{J^3(e^{\beta \abs{x+y}} w(0))}_{L^2}.$ 

Since $w$ is supported in the set $\{(x,y,t)\colon \abs{x+y} \geq R,\, t\in [0,1]\}$ and using that the $\mu_R$ and its derivatives are bounded by a constant independent of $R,$ it follows
\begin{equation*}
	\begin{split}
		c\, \beta^2 \norm{J^3(e^{\beta \abs{x+y}} w(0))}_{L^2}&\leq c\,\beta^5 \sum_{0\leq k+l \leq 3} \norm{e^{\beta \abs{x+y}} \partial_x^k \partial_y^l w(0)}_{L^2}\\
																			&\leq c\, \beta^5 \sum_{0\leq k+l \leq 3} \norm{e^{\beta \abs{x+y}} \partial_x^k \partial_y^l w(0)}_{L^2(\abs{x+y} \geq R\big)}\\
																			&\leq c \,\beta^5 \sum_{0\leq k+l \leq 3} \norm{e^{\beta \abs{x+y}} \partial_x^k \partial_y^l v(0)}_{L^2\big(\abs{x+y}\geq R \big)}.
	\end{split}
\end{equation*}		
Now we want to choose $\beta$ in such a way to obtain in the right-hand side of the previous estimate the weighted norm of $v(0)$ with the  right exponential weight. Let 
\begin{equation*}
	\beta= \frac{2 a R^\frac{3}{2}}{ 18 R -1}.
\end{equation*}
Using the explicit expression of $\beta$ it can be seen that for $R$ sufficiently large depending on $a$ one has
\begin{equation*}
	\beta^5 e^{\beta \abs{x+y}} \leq \Big( \frac{2 a R^\frac{3}{2}}{18R -1} \Big)^5 e^{\frac{2 a R^\frac{3}{2}}{18 R- 1} \abs{x + y}} \leq c_a e^{\frac{a}{8} \abs{x+y}^\frac{3}{2}}, \qquad \text{for} \quad \abs{x + y} \geq R.
\end{equation*}
Using the previous estimate one has
\begin{equation*}
	c \beta^2 \norm{ J^3 (e^{\beta \abs{x+y}} w(0))}_{L^2} \leq c_a \sum_{0\leq k+l \leq 3} \norm{e^{\frac{a}{8} \abs{x+y}^\frac{3}{2}} \partial_x^k \partial_y^l v(0)}_{L^2(\abs{x+y} \geq R)}.
\end{equation*}
Let us recall that under our hypothesis $v(0)\in L^2(e^{a\abs{x+y}^\frac{3}{2}} dx\, dy),$ this can be rephrase saying that 
\begin{equation}\label{hyp}
	\norm{e^{\frac{a}{2} \abs{x+y}^\frac{3}{2}} v(0)}_{L^2}
\end{equation}
is finite.

Using an interpolation argument and the finiteness of~\eqref{hyp}, it can be seen that $\norm{e^{\frac{a}{8} \abs{x+y}^\frac{3}{2}} \partial_x^k \partial_y^l v(0)}_{L^2}$ is finite. 

To show this we will employ the following interpolation result.
\begin{lemma}\label{interpolation}
	For $s>0$ and $a>0,$ let $f\in H^s(\R^2) \cap L^2\big(e^{a\abs{x+y}^\frac{3}{2}} \, dx dy\big).$ Then, for $\theta \in [0,1],$
	\begin{equation*}
		\norm{J^{s(1-\theta)}\big( e^{\theta \frac{a}{2} \abs{x+y}^\frac{3}{2}} f \big)}_{L^2} \leq C \norm{J^s f}_{L^2}^{1-\theta} \norm{e^{\frac{a}{2} \abs{x+y}^\frac{3}{2}} f}_{L^2}^\theta,
	\end{equation*}
	for $C= C(a, s).$
\end{lemma}
Observe that by our hypotheses $v(0)\in L^2(e^{a \abs{x+y}^\frac{3}{2}} dx dy)$ and $v(t) \in C\big([0,1]; H^4(\R^2)\big)$ hence Lemma~\ref{interpolation} with $s=4$ and $\theta=\frac{1}{4}$ ensures that $\norm{e^{\frac{a}{8} \abs{x+y}^\frac{3}{2}} \partial_x^k \partial_y^l v(0)}_{L^2}$ is finite. 

Using this fact we obtain
\begin{equation}\label{one}
	c \, \beta^2 \norm{J^3(e^{\beta \abs{x+y}} w(0))}_{L^2}\leq c_a.
\end{equation}
A similar argument shows that
\begin{equation}\label{two}
	c \, \beta^2 \norm{J^3(e^{\beta \abs{x+y}} w(1))}_{L^2}\leq c_a.
\end{equation}

It remains to bound the third term in the right-hand side of~\eqref{last_one}. 

Since $e_R$ is supported in $\Omega_R:=\{(x,y, t) \colon R\leq \abs{x + y} \leq \frac{18R -1}{4}, \, t\in[0,1]\},$ we find that
\begin{equation}\label{three}
	\begin{split}
		\norm{e^{\beta \abs{x+y}} e_R}_{L_t^1 L_{x y}^2 \cap L_x^1 L_{y t}^2} &\leq e^{\beta \frac{18 R -1}{4}} \norm{e_R \chi_{\Omega_R}}_{L_t^1 L_{x y}^2 \cap L_x^1 L_{y t}^2}\\
																				&\leq c\, e^{\beta \frac{18 R -1}{4}} \norm{(\abs{v} + \abs{\partial_x v} + \abs{\partial_y v} + \abs{\partial_x^2 v} + \abs{\partial_y^2 v}) \chi_{\Omega_R}}_{L_t^1 L_{x y}^2 \cap L_x^1 L_{y t}^2}\\
																				&\leq c R^\frac{1}{2} e^{\beta \frac{18 R -1}{4}},
	\end{split}
\end{equation} 
where in the last inequality we have used H\"older inequality and the fact that the area of the region $\Omega_R$ is of order $R.$

Summing up, using~\eqref{one},\eqref{two} and~\eqref{three} we have
\begin{equation*}
	\norm{e^{\beta \abs{x+y}} w}_{L^2} + \sum_{0< k+ l \leq 2} \norm{e^{\beta \abs{x+y}} \partial_x^k \partial_y^l w}_{L_x^\infty\, L_{y t}^2}\leq c_a + c R^\frac{1}{2} e^{\beta \frac{18 R- 1}{4}}\leq c_a R^\frac{1}{2} e^{\beta \frac{18 R- 1}{4}}.
\end{equation*}
Defining $D_R:= \{18R-1 \leq \abs{x+y} \leq 18R \, \land \, 18R-1 \leq \abs{x-y} \leq 18R \}$ see that $D_R\times[0,1]\subset \{\abs{x+y} \geq R\}\times[0,1],$ the set in which $w$ is supported. Observing that in $D_R\times[0,1]$ we have $w=v,$ one obtains
\begin{equation*}
	\begin{split}
		\norm{e^{\beta \abs{x+y}} v}_{L^2(D_R\times[0,1])} + \sum_{0< k+l \leq 2} \norm{e^{\beta \abs{x+y}} \partial_x^k \partial_y^l v}_{L^2(D_R\times[0,1])} &\leq R^\frac{1}{2} \big( \norm{e^{\beta \abs{x+y}} w}_{L^2} + \sum_{0< k+ l \leq 2} \norm{e^{\beta \abs{x+y}} \partial_x^k \partial_y^l w}_{L_x^\infty\, L_{y t}^2}  \big)\\
								&\leq c_a R e^{\beta \frac{18R -1}{4}}.	
	\end{split}
\end{equation*}
If $\abs{x + y} \geq 18 R-1,$ then 
\begin{equation*}
	\beta \abs{x+y} \geq \beta (18 R -1) = 2 a R^\frac{3}{2}.
\end{equation*}
Moreover since for sufficiently large $R,$\, $\beta \geq 1,$ one gets
\begin{equation*}
	R\, e^{\beta \frac{18 R-1}{4}} \leq e^{(\beta + 1)\frac{18 R-1}{4}} \leq e^{\beta (1 + \frac{1}{\beta}) \frac{18 R-1}{4}} \leq e^{2 \beta\frac{18 R-1}{4}}= e^{a R^\frac{3}{2}}. 
\end{equation*}
This implies that
\begin{equation*}
	e^{2 a R^\frac{3}{2}} \big( \norm{v}_{L^2(D_R\times[0,1])} + \sum_{0< k+l \leq 2} \norm{\partial_x^k \partial_y^l v}_{L^2(D_R\times[0,1])} \big) \leq c_a e^{a R^\frac{3}{2}},
\end{equation*}
which is equivalent to
\begin{equation*}
	\norm{v}_{L^2(D_R\times[0,1])} + \sum_{0< k+l \leq 2} \norm{\partial_x^k \partial_y^l v}_{L^2(D_R\times[0,1])}\leq c_a e^{- a R^\frac{3}{2}},
\end{equation*}
which written in terms of $Q_R$ gives
\begin{equation*}
	\norm{v}_{L^2(Q_R\times[0,1])} + \sum_{0< k+l \leq 2} \norm{\partial_x^k \partial_y^l v}_{L^2(Q_R\times[0,1])}\leq c_a e^{- a \big(\frac{R}{18}\big)^\frac{3}{2}}.
\end{equation*}
which yields the desired upper bound.
\end{proof}


\subsection{Conclusion of the proofs}

As it was commented in the introduction Theorem~\ref{main_result_s} will follow as a consequence of Theorem~\ref{linear_result}. Therefore we first provide the proof of Theorem~\ref{linear_result}.

\subsubsection{Proof of Theorem~\ref{linear_result}}
		If $v\nequiv 0$ we can assume after a possible translation, dilation and multiplication by a constant that $v$ satisfies the hypotheses of Theorem~\ref{lower_bound}. This means that for $R$ sufficiently large there exist constants $c_0$ and $c_1$ as given in Theorem~\ref{lower_bound} such that  
		\begin{equation}\label{upper_estimate}
			A_R(v)\geq c_0 e^{-c_1 R^\frac{3}{2}},
		\end{equation}
		where
		\begin{equation*}
		A_R(v):= \Bigg( \iint_{\,Q_R\times [0,1]} \big(\abs{v}^2 + \abs{\nabla v}^2 + \abs{\Delta v}^2\big)\, dx\, dy\, dt \Bigg)^\frac{1}{2}
	\end{equation*}
	and $Q_R=\{(x,y)\colon R-1\leq \abs{x+y}\leq R\, \land \, R-1\leq \abs{x-y}\leq R\}.$
		
		Applying Theorem~\ref{upper_bound} we can conclude that
		\begin{equation*}
			\norm{v}_{L^2(Q_R \times [0,1])} + \sum_{ 0< k+l\leq 2} \norm{\partial_x^k \partial_y^l v}_{L^2(Q_R \times [0,1])}\leq c e^{-a \big(\frac{R}{18}\big)^\frac{3}{2}}.
		\end{equation*}

		It is easy to see that the left-hand side of the previous expression can be bounded from below by the quantity $A_R(v),$ this gives
		\begin{equation}\label{lower_estimate}
			A_R(v)\leq c e^{-\frac{a}{18^{3/2}} R^\frac{3}{2}}.
		\end{equation}
		If one assumes $a>a_0:=18^{3/2} c_1,$ combining~\eqref{upper_estimate} and~\eqref{lower_estimate} and taking the limit as $R$ tends to infinity we get a contradiction.
		
		Therefore $v\equiv 0$ and Theorem~\ref{linear_result} is proved. \qed
		
	\subsubsection{Proof of Theorem~\ref{main_result_s}}
	We just need to show that Theorem~\ref{linear_result} applies when we consider $v$ as the difference $u_1 - u_2$ of two solutions to~\eqref{Z-K_s}.
			
	First of all we have already shown that if $u_1$ and $u_2$ are solutions to~\eqref{Z-K_s} then the difference $v$ satisfies
			\begin{equation*}
				\partial_t v + (\partial_x^3 + \partial_y^3) v + a_1 (\partial_x + \partial_y) v + a_0 v=0,
			\end{equation*}
			where 
			\begin{equation}\label{a_0-a_1}
				a_0= 4^{-\frac{1}{3}} (\partial_x + \partial_y)u_2 \qquad \text{and} \qquad a_1=4^{-\frac{1}{3}} u_1.
			\end{equation}
			Thus, one just needs to check that the coefficients $a_0, a_1$ as defined above satisfy the assumptions of Theorem~\ref{linear_result}, that is 
			\begin{equation}\label{cond:a_0-a_1}
			a_0 \in L^\infty \cap L_x^2L_{y,t}^\infty \qquad a_1 \in L^\infty \cap L_x^2 L_{y t}^\infty \cap L_x^1 L_{y t}^\infty.
			\end{equation}

	To do that, proceeding as in~\cite{B_I_M}, we will use the following interpolation result(see~\cite{N_P}).
	\begin{lemma}
		For $s>0$ and $a>0,$ let $f \in H^s(\R^2) \cap L^2( (1+ \abs{x + y})^{2a} \, dx dy).$ 
		Then for any $\theta \in (0,1),$
		\begin{equation}\label{interpolation_2}
			\norm{J^{\theta s} \big( (1+ \abs{x+y})^{(1-\theta) a} f \big)}_{L^2}\leq C \norm{J^s f}_{L^2}^{\theta} \norm{(1+ \abs{x + y})^a f}_{L^2}^{1-\theta},
		\end{equation}
		for $C=C(a,s).$
	\end{lemma}
	Applying~\eqref{interpolation_2} with $s=4, a=\frac{4}{3} + \varepsilon$ and $\theta= \frac{1}{4} + \frac{3}{16}\varepsilon$ with $\varepsilon$ as in the statement of Theorem~\ref{main_result_s}, we have
	\begin{equation}\label{interpolation_explicit}
		\norm{J^{1+ \frac{3}{4}\varepsilon} \big( (1+ \abs{x + y})^{(1+ \varepsilon_1)} f \big)}_{L^2}\leq C \norm{J^4 f}_{L^2}^{\theta} \norm{(1+ \abs{x + y})^{\left(\frac{4}{3} + \varepsilon \right)} f}_{L^2}^{1-\theta},
	\end{equation}
	where $\varepsilon_1:= \frac{\varepsilon}{2} - \frac{3}{16}\varepsilon^2>0.$
	
	Applying~\eqref{interpolation_explicit} with $f=a_1= 4^{-\frac{1}{3}} u_1(t),$ from our hypothesis about the solution $u_1$ and from the embedding $H^{1 + \frac{3}{4} \varepsilon}(\R^2)\hookrightarrow L^\infty(\R^2) \cap C(\R^2)$ we obtain
	\begin{equation}\label{first_decay}
		\abs{u_1(x,y,t)}\leq \frac{c}{(1+ \abs{x+ y})^{(1+\varepsilon_1)}},
	\end{equation}
	for all $(x, y, t) \in \R^2 \times [0,1].$
	
	Since $1 + \frac{3}{4}\varepsilon >1,$ the estimate~\eqref{interpolation_explicit} is also true for $J^1$ instead of $J^{1 + \frac{3}{4}\varepsilon}$ with $f=4^{-\frac{1}{2}} u_2,$ using the product rule for the derivatives we obtain that $\norm{(1+\abs{x+y})^{(1+ \varepsilon_1)} 4^{-\frac{1}{3}} \partial_x u_2(t)}_{L^2(\R^2)}$ and  $\norm{(1+\abs{x+y})^{(1+ \varepsilon_1)} 4^{-\frac{1}{3}} \partial_y u_2(t)}_{L^2(\R^2)}$ are bounded function of $t\in [0,1].$ This allows us to apply~\eqref{interpolation_2} with $f=4^{-\frac{1}{3}} \partial_x u_2(t)$ and $f=4^{-\frac{1}{3}} \partial_y u_2(t),$ $s=3, a= 1+ \varepsilon_1$ and $\theta=\frac{1}{3} + \varepsilon_2$ with $\varepsilon_2>0$ small to obtain
	\begin{equation*}
		\norm{J^{1+ 3\varepsilon_2} \big( (1+ \abs{x + y})^{ \frac{2}{3}} 4^{-\frac{1}{3}} \partial_x u_2(t) \big)}_{L^2}\leq C \norm{J^3 (4^{-\frac{1}{3}} \partial_x u_2(t))}_{L^2}^{\theta} \norm{(1+ \abs{x + y})^{(1 + \varepsilon_1)} 4^{- \frac{1}{3}} \partial_x u_2(t)}_{L^2}^{1-\theta}
	\end{equation*}
	and the same for the derivative with respect to $y.$
	
	Using this estimate and again the Sobolev embeddings one has
	\begin{equation}\label{second_decay}
		\abs{4^{- \frac{1}{3}} (\partial_x + \partial_y) u_2(t)} \leq \frac{c}{(1+ \abs{x+y})^\frac{2}{3}}
	\end{equation}
	for all $(x,y,t)\in \R^2 \times [0,1].$
	
	From the decay properties expressed in~\eqref{first_decay} and~\eqref{second_decay} it is clear that hypothesis~\eqref{cond:a_0-a_1} holds.
	
	\begin{remark}
	As a final remark we observe that from~\eqref{first_decay} and~\eqref{second_decay} the functions $\widetilde{a}_j:=a_j(x,y,t)\widetilde{\mu}_R,$ with $j=0,1$ and $a_j(x,y,t)$ as in~\eqref{a_0-a_1}, have small norms in the corresponding spaces for $R$ sufficiently large as required in Theorem~\ref{upper_bound}. 
	
	Indeed choosing $\widetilde{\mu}_R$ in Theorem~\ref{upper_bound} as
	\begin{equation*}
		\widetilde{\mu}_R(x,y)=\chi_{\{(x,y): \abs{x + y} \geq R\}}(x,y),
	\end{equation*}
	with $\chi_{\Omega}$ the indicator function of the set $\Omega,$ it is easy to see from~\eqref{first_decay} and~\eqref{second_decay} that the following four terms  
	\begin{equation*}
		\norm{a_0 \chi_{\{(x,y): \abs{x+y} \geq R\}}}_{L^\infty \cap L_x^2 L_{y t}^\infty}, \qquad  \norm{a_1 \chi_{\{(x,y): \abs{x + y} \geq R\}}}_{L_x^2 L_{y t}^\infty \cap L_x^1 L_{y t}^\infty}
	\end{equation*}
	tend to zero as $R$ tends to infinity. This guarantees the smallness required. 
	\end{remark}
\qed	


\section{Proof of Theorem~\ref{main_result}}

As a starting point we recall that if $u_j=u_j(x,y,t),\, j=0,1$ is a solution to~\eqref{Z-K}, then $\widetilde{u}_j=\widetilde{u}_j(x',y',t):= u_j\big(\frac{x'+y'}{2\mu}, \frac{x'-y'}{2\lambda},t\big),\, j=0,1$ satisfies the symmetric problem~\eqref{Z-K_s}. In particular, this entails that if one provides uniqueness for solutions to~\eqref{Z-K_s}, namely $\widetilde{u}_1\equiv \widetilde{u}_2,$ then uniqueness for solutions to~\eqref{Z-K}, namely $u_1\equiv u_2,$ is also given. 
Thus in order to get our result, one just needs to show that if $u_j,\, j=0,1$ satisfies the hypotheses in Theorem~\ref{main_result}, then Theorem~\ref{main_result_s} applies to $\widetilde{u}_j,\, j=0,1.$ 

Using~\eqref{change_variable}, by elementary change of integration variables, it can be seen that for $j=0,1$
\begin{multline*}
	\norm{u_1(t_j)-u_2(t_j)}_{L^2(e^{a \abs{x}^{3/2}}\, dx dy)}^2= \iint_{\,\R^2} e^{a\abs{x}^\frac{3}{2}} (u_1(t_j)-u_2(t_j))^2\, dx\,dy\\= \frac{1}{2\lambda\mu} \iint_{\,\R^2} e^{\frac{a}{(2\mu)^{3/2}} \abs{x'+y'}^{3/2}} (\widetilde{u}_1(t_j)-\widetilde{u}_2(t_j))^2\, dx'\,dy'= \frac{1}{2\lambda\mu} \norm{\widetilde{u}_1(t_j)-\widetilde{u}_2(t_j)}_{L^2(e^{a/(2\mu)^{3/2} \abs{x+y}^{3/2}}\, dx dy)}^2, 
\end{multline*}
where we have used the notation $t_0=0,\, t_1=1.$ 

From the previous trivial identity it follows that if $u_1-u_2$ satisfies the decay assumption~\eqref{decay_properties} then 
\begin{equation*}
	\widetilde{u}_1(0)-\widetilde{u}_2(0), \widetilde{u}_1(1)-\widetilde{u}_2(1) \in L^2\big( e^{\widetilde{a} \abs{x+y}^{\frac{3}{2}}} dx dy\big),
\end{equation*} 
with $\widetilde{a}:=a/(2\mu)^{3/2}.$

Therefore, by Theorem~\ref{main_result_s}, we conclude that there exists a universal constant $a_0$ such that if $\widetilde{a}>a_0,$ that is if $a>(2\mu)^{3/2} a_0,$ then $\widetilde{u}_1\equiv \widetilde{u}_2.$ In particular this implies $u_1\equiv u_2,$ which is the desired result.
\qed

\section{Proof of Theorem~\ref{thm:optimality}}
This section is devoted to the proof of the optimality of our unique continuation result Theorem~\ref{main_result}. 

Following the argument provided in~\cite{I_L_P} we define the weight
\begin{equation*}
	\varphi_n(z,t):=
	\begin{system}
		e^{a(t)/4} \qquad \hspace{0.25cm} z\leq 0\\
		e^{a(t)\theta(z)} \qquad 0\leq z\leq 1\\
		e^{a(t) z^{3/2}} \qquad 1\leq z\leq n\\
		P_2(z,t) \qquad z\geq n
	\end{system}
\end{equation*}
where $a(t)$ is the unique solution to the following IVP
\begin{equation}\label{ODE_a}
	\begin{system}
		a'(t) + \frac{27}{4} a^3(t)=0,\\
		a(0)=a_0,
	\end{system}
\end{equation}
namely 
\begin{equation}\label{eq:a(t)}
	a(t)= \frac{a_0}{(1+27a_0^2t/2)^{1/2}}, \qquad \forall\, t\geq 0,
\end{equation}

\begin{equation*}
	\theta(z)= \frac{1}{4} + \frac{15}{8}z^3 - \frac{12}{8}z^4 + \frac{3}{8}z^5
\end{equation*}
and 
\begin{equation*}
	P_2(z,t)= e^{a(t)n^{3/2}} + \frac{3}{2}a(t)n^{1/2} e^{a(t)n^{3/2}}(z-n) + \Big( \big(\frac{3}{2} a(t) n^{1/2}\big)^2 + \frac{3}{4} a(t) n^{-1/2} \Big) e^{a(t)n^{3/2}} \frac{(z-n)^2}{2},
\end{equation*}
that is the second degree truncated Taylor expansion at $n$ of $e^{a(t)z^{3/2}}.$

The following properties will be useful hereafter.
\begin{itemize}
	\item For any $n\in \mathbb{N}$ and $z\geq 0$ one has
	\begin{equation}\label{phi_n_estimate}
	\varphi_n(z,t)\leq C_{a_0}e^{a(t)z^{3/2}},
\end{equation}
	for all $t\geq 0.$
	\item The function $a(t)\in (0,a_0]$ for all $t\geq 0.$
	\item Being 
	\begin{equation*}
		\theta''(z)=\frac{3}{4} z \Big( \big( \sqrt{10} z - \frac{12}{\sqrt{10}} \big)^2 + \frac{3}{5} \Big) \geq 0,
	\end{equation*}
	and $\theta'(0)=0,$ then $\theta'(z)\geq 0$ for $0\leq z\leq 1.$
	\item 
	\begin{equation*}
		\partial_z P_2(z,t)= \frac{3}{2} a(t) n^{1/2} e^{a(t)n^{3/2}} + \Big( \big(\frac{3}{2} a(t) n^{1/2} \big)^2 + \frac{3}{4} a(t) n^{-1/2}\Big) e^{a(t)n^{3/2}}(z-n) \geq 0,
	\end{equation*}
	for $z\geq n$ and $t\geq 0.$
\end{itemize}
From the previous facts, it follows that
\begin{equation*}
	\partial_z \varphi_n(z,t)\geq 0, \qquad (z,t)\in \R \times[0,\infty).
\end{equation*}
Let us define 
\begin{equation*}
	\phi_n(x,y,t):= \varphi_n(x+y,t).
\end{equation*}

Observe that 
\begin{equation}\label{positivity}
	\partial_x \phi_n(x,y,t)=\partial_y \phi_n(x,y,t)=\partial_z \varphi_n(x+y,t)\geq 0.
\end{equation}

Next, we multiply equation~\eqref{Z-K_s} by $u \phi_n,$ integrating the resulting identity and using integration by parts we get
\begin{equation*}
	\begin{split}
		\frac{1}{2} \frac{d}{dt} \int_{\R^2} u^2\, \phi_n - \frac{1}{2} \int_{\R^2} u^2 \partial_t \phi_n &-\frac{1}{2} \int_{\R^2} u^2\, \partial_x^3 \phi_n +\frac{3}{2} \int_{\R^2} (\partial_x u)^2 \partial_x \phi_n -\frac{4^{-1/3}}{3}\int_{\R^2} u^3 \partial_x \phi_n\\
		&-\frac{1}{2} \int_{\R^2} u^2\, \partial_y^3 \phi_n +\frac{3}{2} \int_{\R^2} (\partial_y u)^2 \partial_y \phi_n -\frac{4^{-1/3}}{3}\int_{\R^2} u^3 \partial_y \phi_n=0.
	\end{split}
\end{equation*} 
Multiplying the last identity by $2$ and using~\eqref{positivity} we can drop the positive terms to obtain
\begin{equation}\label{final_optimality}
	\frac{d}{dt} \int_{\R^2} u^2 \phi_n\leq \int_{\R^2} u^2 \big(\partial_x^3 \phi_n + \partial_y^3 \phi_n + \partial_t \phi_n \big) + \frac{2}{3}4^{-1/3} \int_{\R^2} u^3 \big(\partial_x \phi_n + \partial_y \phi_n\big).
\end{equation}

\begin{remark}
	Notice that by virtue of our assumptions, $u$ satisfies the hypotheses of Theorem~\ref{persistence_properties}, this means that in particular for any $t\in [0,1]$ $u(t)\in H^2(e^{2\beta(x+y)}\, dx dy)$ for all $\beta>0.$ This fact allows us to justify the integration by parts used to obtain~\eqref{final_optimality}. Indeed at infinity $\varphi_n(z,t)$ as a function of $z$ is a polynomial of order two. 
\end{remark}
%

Let us consider the right-hand side of~\eqref{final_optimality} in four different domains, namely 
\begin{equation*}
	a)\, x+y\leq 0,\qquad b)\, 0\leq x+y\leq 1,\qquad c)\, 1\leq x+y\leq n,\qquad d)\, x+y\geq n.
\end{equation*}
\begin{enumerate}[$a)$]
\item In the region $x+y\leq 0$ we have 
\begin{equation*}
	\partial_x^j \phi_n(x,y,t)=\partial_y^j \phi_n(x,y,t)=0, \quad j=1,2,3
\end{equation*}
and 
\begin{equation*}
	\partial_t \phi_n(x,y,t)= \frac{a'(t)}{4}e^{a(t)/4}=-\frac{27}{16}a^3(t)e^{a(t)/4}\leq 0.
\end{equation*}
Therefore
\begin{equation}\label{min_zero}
	\int_{\{x+y\leq 0\}} u^2(\partial_x^3 \phi_n + \partial_y^3 \phi_n + \partial_t \phi_n) + \frac{2}{3} 4^{-1/3} \int_{\{x+y\leq 0\}} u^3(\partial_x \phi_n + \partial_y \phi_n)\leq 0
\end{equation}
and hence the region $x+y\leq 0$ does not give any contribution to the right-hand side of~\eqref{final_optimality}. 
\item In the domain $0\leq x+y\leq 1$ we have 
\begin{equation*}
	\partial_t \phi_n(x,y,t)=a'(t)\theta(x+y)e^{a(t)\theta(x+y)}\leq 0,
\end{equation*}
indeed $a'(t)=-27/4a^3(t)\leq 0$ and $\theta(x+y)\geq \theta(0)=1/4>0.$

Moreover
\begin{equation*}
	\partial_x \phi_n(x,y,t)=\partial_y \phi_n(x,y,t)=a(t)\theta'(x+y)\phi_n(x,y,t)\leq c a_0 \phi_n(x,y,t),
\end{equation*}

\begin{equation*}
	\begin{split}
	\partial_x^2 \phi_n(x,y,t)&=\partial_y^2 \phi_n(x,y,t)=\big(a(t) \theta''(x+y) + (a(t)\theta'(x+y))^2\big)\phi_n(x,y,t)\\
	& \leq c(a_0 + a_0^2)\phi_n(x,y,t),
	\end{split}
\end{equation*}

\begin{equation*}
	\begin{split}
	\partial_x^3 \phi_n(x,y,t)&=\partial_y^3 \phi_n(x,y,t)=\big(a(t) \theta^{(3)}(x+y) + 3 a^2(t) \theta'(x+y) \theta''(x+y) + (a(t)\theta'(x+y))^3\big)\phi_n(x,y,t)\\
	& \leq c(a_0 + a_0^2 + a_0^3)\phi_n(x,y,t).
	\end{split}
\end{equation*}

Therefore 
\begin{equation}\label{zero_one}
	\int_{\{0\leq x+y\leq 1\}} u^2(\partial_x^3\phi_n + \partial_y^3 \phi_n +\partial_t \phi_n)\leq C_{a_0} \int_{\{0\leq x+y\leq 1\}} u^2 \phi_n
\end{equation}
and
\begin{equation}\label{zero_one_bis}
	\begin{split}
	\frac{2}{3}4^{-1/3} \int_{\{0\leq x+y\leq 1\}} u^3(\partial_x \phi_n + \partial_y \phi_n)&\leq C_{a_0} \norm{u(t)}_{L^\infty(0\leq x+y\leq 1)} \int_{\{0\leq x+y\leq 1\}} u^2 \phi_n\\
	&\leq C_{a_0} \norm{u}_{C([0,1];H^2(\R^2))}\int_{\{0\leq x+y\leq 1\}} u^2 \phi_n\\
	&\leq C_{a_0,u} \int_{\{0\leq x+y\leq 1\}} u^2 \phi_n,
	\end{split}
\end{equation}
where in the last but one inequality we have used the Sobolev embedding.

Notice that under our hypothesis about the solution $u$ the norm $\norm{u}_{C([0,1]; H^2(\R^2))}$ is finite.

\item In the region $1\leq x+y\leq n$ we have
\begin{equation*}
	\partial_t \phi_n(x,y,t)=a'(t)(x+y)^{3/2}\phi_n(x,y,t).
\end{equation*}

Moreover
\begin{equation*}
	\partial_x \phi_n(x,y,t)=\partial_y \phi_n(x,y,t)=\frac{3}{2} a(t)(x+y)^{1/2}\phi_n(x,y,t),
\end{equation*}

\begin{equation*}
	\partial_x^2 \phi_n(x,y,t)=\partial_y^2 \phi_n(x,y,t)=\Big(\frac{3}{4} a(t) (x+y)^{-1/2} + \frac{9}{4} a^2(t) (x+y)\Big)\phi_n(x,y,t)
\end{equation*}

\begin{equation*}
	\partial_x^3 \phi_n(x,y,t)=\partial_y^3 \phi_n(x,y,t)=\Big(-\frac{3}{8} a(t) (x+y)^{-3/2} + \frac{27}{8} a^2(t) +\frac{27}{8} a^3(t)(x+y)^{3/2}\Big)\phi_n(x,y,t)
\end{equation*}

Therefore 
\begin{equation*}
	\partial_x^3 \phi_n + \partial_y^3 \phi_n +\partial_t \phi_n=\Big( -\frac{3}{4} a(t) (x+y)^{-3/2} + \frac{27}{4} a^2(t) +\frac{27}{4} a^3(t)(x+y)^{3/2} + a'(t)(x+y)^{3/2}\Big)\phi_n.
\end{equation*}

Let us observe that the first term of the right hand side of the previous identity is negative, therefore, using also that $a(t)$ solves the Cauchy problem~\eqref{ODE_a} we get
\begin{equation*}
	\partial_x^3 \phi_n + \partial_y^3 \phi_n + \partial_t \phi_n \leq \frac{27}{4}a^2(t)\phi_n \leq \frac{27}{4}a_0^2 \phi_n.
\end{equation*} 
Using the previous inequality we have
\begin{equation}\label{one_n}
	\int_{\{1\leq x+y\leq n\}} u^2(\partial_x^3\phi_n + \partial_y^3 \phi_n +\partial_t \phi_n)\leq \frac{27}{4}a_0^2 \int_{\{1\leq x+y\leq n\}} u^2 \phi_n.
\end{equation}
and moreover
\begin{equation}\label{one_n_bis}
	\begin{split}
	\frac{2}{3}4^{-1/3} \int_{\{1\leq x+y\leq n\}} u^3(\partial_x \phi_n + \partial_y \phi_n)&\leq C a(t)\int_{\{1\leq x+y\leq n\}} (x+y)^{1/2} u^3 \phi_n\\
	& \leq C_{a_0} \norm{(x+y)^{1/2} u(t)}_{L^\infty(1\leq x+y\leq n)} \int_{\{1\leq x+y\leq n\}} u^2 \phi_n\\
	&\leq C_{a_0} \norm{e^{x+y} u}_{C([0,1]; H^2(\R^2))} \int_{\{1\leq x+y\leq n\}} u^2 \phi_n\\
	&\leq C_{a_0,u} \int_{\{1\leq x+y\leq n\}} u^2 \phi_n,
	\end{split}
\end{equation}
where in the last but one inequality we have used that $\norm{(x+y)^{1/2}u(t)}_{L^\infty(1\leq x+y\leq n)}\leq \norm{e^{x+y}u(t)}_{L^\infty(\R^2)}$ and the Sobolev embedding.

Moreover notice that Theorem~\ref{persistence_properties} guarantees that $\norm{e^{x+y}u}_{C([0,1]; H^2(\R^2))}$ is finite.
\item In the domain $x+y\geq n$ we have
\begin{equation*}
	\partial_x^3 \phi_n(x,y,t)=\partial_x^3 P_2(x+y,t)=0, \qquad \partial_y^3 \phi_n(x,y,t)=\partial_y^3 P_2(x+y,t)=0.
\end{equation*}
Moreover
\begin{equation*}
	\partial_t \phi_n(x,y,t)=\partial_t P_2(x+y,t)=a'(t)[\cdot]e^{a(t)n^{3/2}}\leq 0.
\end{equation*}
We also have that for $x+y\geq n$
\begin{equation*}
	\begin{split}
		\partial_x \phi(x,y,t)&= \partial_x P_2(x+y,t)= \frac{3}{2}a(t) n^{1/2}e^{a(t)n^{3/2}} + \frac{3}{2}a(t)n^{1/2} \Big[\frac{3}{2}a(t)n^{1/2} + \frac{1}{2n} \Big]e^{a(t)n^{3/2}}(x+y-n)\\
		&\leq \frac{3}{2}a(t) n^{1/2} P_2(x+y,t) + \Big[\frac{3}{2}a(t)n^{1/2} + \frac{1}{2n} \Big] P_2(x+y,t)\\
		&\leq (1+3a_0 n^{1/2})P_2(x+y,t)\\
		&\leq (1+3a_0 (x+y)^{1/2})\phi_n(x,y,t),
	\end{split}
\end{equation*}
and in the same way
\begin{equation*}
	\partial_y \phi(x,y,t)= \partial_y P_2(x+y,t)\leq (1+3a_0 (x+y)^{1/2})\phi_n(x,y,t).
\end{equation*}
Therefore
\begin{equation}\label{greater_n}
	\int_{\{x+y\geq n\}} u^2(\partial_x^3\phi_n + \partial_y^3 \phi_n +\partial_t \phi_n)\leq 0.
\end{equation}
proceeding as in the previous domain we also have 
\begin{equation}\label{greater_n_bis}
	\begin{split}
	\frac{2}{3}4^{-1/3} \int_{\{x+y\geq n\}} u^3(\partial_x \phi_n + \partial_y \phi_n)&\leq C \int_{\{x+y\geq n\}} u^3 (1+3a(t)(x+y)^{1/2}) \phi_n\\
	& \leq C_{a_0} \norm{(x+y)^{1/2} u(t)}_{L^\infty(x+y\geq n)} \int_{\{x+y\geq n\}} u^2 \phi_n\\
	&\leq C_{a_0} \norm{e^{x+y} u}_{C([0,1]; H^2(\R^2))} \int_{\{x+y\geq n\}} u^2 \phi_n\\
	&\leq C_{a_0,u} \int_{\{x+y\geq n\}} u^2 \phi_n.
	\end{split}
\end{equation}
\end{enumerate}

Using~\eqref{min_zero}-~\eqref{greater_n_bis} in~\eqref{final_optimality} we get
\begin{equation*}
	\frac{d}{dt}\int_{\R^2} u^2 \phi_n\leq C_{a_0,u} \int_{\R^2} u^2 \phi_n.
\end{equation*}
Applying the Gronwall inequality we obtain
\begin{equation*}
	\int_{\R^2} u^2(t) \phi_n \leq e^{C_{a_0,u} t}\int_{\R^2} u^2(0)\phi_n(0) \qquad \forall\, t\in [0,1].
\end{equation*}
The conclusion follows using~\eqref{phi_n_estimate} at $t=0$ and by Fatou's lemma letting $n$ go to infinity. 

\appendix
\section{Proof of Lemma~\ref{lemma2.2bis}}
Now we are in position to prove Lemma~\ref{lemma2.2bis}.
Actually we will give a proof of a slightly different and more general version of the previous lemma. Our result Lemma~\ref{lemma2.2bis} follows by using the same argument. 

\begin{lemma}\label{lemma2.2}
Let $w \in C([0,1];H^4(\R^2))\cap C^1([0,1]; L^2(\R^2))$ such that for all $t\in[0,1]$ $\supp w(t)\subseteq K,$ where $K$ is a compact subset of $\R^2.$
	
	Assume that $a_0\in L^\infty \cap L_x^2 L_{y t}^\infty$ and $a_1 \in L_x^2 L_{y t}^\infty\cap L_x^1 L_{y t}^\infty,$ with small norms in these spaces. 
	
	Then there exists $c>0,$ independent of the set $K,$ such that for $\beta >0$ and $\lambda >0$ the following estimate holds
	\begin{multline*}
		\norm{e^{\lambda \abs{x}} e^{\beta \abs{y}} w}_{L^2(\R^2\times [0,1])} + \sum_{0<k+l\leq 2} \norm{e^{\lambda \abs{x}} e^{\beta \abs{y}} \partial_x^k \partial_y^l w}_{L_x^\infty L_{y t}^2(\R^2\times [0,1])}\\
										\leq c (\lambda^2 + \beta^2) \big(\norm{J^3(e^{\lambda \abs{x}}e^{\beta \abs{y}} w(0))}_{L^2(\R^2)} + \norm{J^3(e^{\lambda \abs{x}}e^{\beta \abs{y}} w(1))}_{L^2(\R^2)}\big)\\
											 + c \norm{e^{\lambda \abs{x}} e^{\beta \abs{y}}(\partial_t + \partial_x^3 + \partial_y^3 + a_1(\partial_x + \partial_y) + a_0) w}_{L_t^1 L_{x y}^2 \cap L_x^1 L_{yt}^2 (\R^2 \times [0,1])}
	\end{multline*}
	with $J$ such that $\widehat{J g}(\xi, \eta):= (1 + \xi^2 + \eta^2)^\frac{1}{2} \widehat{g}(\xi, \eta).$ (Here, $\widehat{\phantom{g}}$ denotes the spatial Fourier transform in $\R^2$ and $(\xi,\eta)$ are the variables in the frequency space corresponding to the space variables $(x,y).$)
\end{lemma}

As in our proof of the Carleman estimate for the operator $P=\partial_t + \partial_x^3 + \partial_y^3 +a_1(\partial_x+\partial_y) +a_0,$ we first prove a counterpart of Lemma~\ref{lemma2.2} for the leading part of the operator $P,$ namely $\partial_t + \partial_x^3 + \partial_y^3.$

\begin{lemma}\label{lemma2.1}
	Let $w \in C([0,1]; H^4(\R^2)) \cap C^1([0,1]; L^2(\R^2))$ such that for all $t \in [0,1]$ $\supp w(t)\subseteq K,$ where $K$ is a compact subset of $\R^2.$ Then
		\begin{enumerate}
			\item For $\lambda>0$ and $\beta>0,$
						\begin{equation}\label{no_deriv}
							\begin{split}
								\norm{e^{\lambda \abs{x}} e^{\beta \abs{y}} w}_{L_t^\infty\, L_{x y}^2(\R^2 \times [0,1])} \leq& \norm{e^{\lambda \abs{x}} e^{\beta \abs{y}} w(0)}_{L^2(\R^2)} + \norm{e^{\lambda \abs{x}} e^{\beta \abs{y}} w(1)}_{L^2(\R^2)}\\
									 &+ \norm{e^{\lambda \abs{x}}  e^{\beta \abs{y}} (\partial_t + \partial_x^3 + \partial_y^3)w}_{L_t^1 L_{x y}^2(\R^2 \times [0,1])}.
							\end{split}
						\end{equation}
			\item There exists $c>0,$ independent of the set $K,$ such that for $\beta\geq 1$ and $\lambda\geq 1$
						\begin{equation}\label{deriv}
							\begin{split}
								\quad\norm{e^{\lambda \abs{x}} e^{\beta \abs{y}} Lw}_{L_x^\infty\, L_{y t}^2(\R^2 \times [0,1])}
								 &\leq c\, (\lambda^2 + \beta^2)\big (\norm{J^3(e^{\lambda \abs{x}} e^{\beta \abs{y}} w(0))}_{L^2(\R^2)} + \norm{J^3(e^{\lambda \abs{x}} e^{\beta \abs{y}} w(1))}_{L^2(\R^2)} \big)\\
									&\phantom{~\leq} + \norm{e^{\lambda \abs{x}} e^{\beta \abs{y}} (\partial_t + \partial_x^3 + \partial_y^3)w}_{L_x^1 L_{y t}^2(\R^2 \times [0,1])},
								\end{split}	
						\end{equation}
						where $L$ denotes any operator in the set $\{ \partial_x, \partial_y, \partial_x^2, \partial_y^2\}$ and $J$ is such that $\widehat{J g}(\xi, \eta):= (1 + \xi^2 + \eta^2)^\frac{1}{2} \widehat{g}(\xi, \eta).$
		\end{enumerate}
		\begin{remark}
			It is a fundamental fact that in order to obtain from~\eqref{no_deriv} and~\eqref{deriv} the estimate in Lemma~\ref{lemma2.2} for the whole operator $P,$ the coefficient in front of the term on the right-hand side of~\eqref{deriv} involving the operator $\partial_t +\partial_x^3 +\partial_y^3$ does not depend on $\lambda$ and $\beta,$ indeed otherwise, since $\lambda$ and $\beta$ grow as $R,$ the correction terms coming from the addition of the lower order derivatives  cannot be hidden in the left-hand side as desired.  	 \end{remark}
\end{lemma}

Before proving Lemma~\ref{lemma2.1} we introduce the following notations.
\begin{equation}\label{H_lb}
	H_{\lambda, \beta}\cdot:= e^{\lambda x} e^{\beta y} (\partial_t + \partial_x^3 + \partial_y^3) e^{-\lambda x} e^{-\beta y}\cdot =\big[\partial_t + (\partial_x - \lambda)^3 + (\partial_y -\beta)^3 \big]\cdot.
\end{equation}
It is easy to see from the previous definition that $H_{\lambda, \beta}$ is defined through the space-time Fourier transform by the multiplier 
\begin{equation*}
	i \tau  + (i \xi - \lambda)^3 + (i \eta - \beta)^3.
\end{equation*}
We can define the inverse operator $T_0$ of $H_{\lambda, \beta}$ by the symbol
\begin{equation}\label{symbol}
	m_0(\xi, \eta, \tau):= \frac{1}{i \tau  + (i \xi - \lambda)^3 + (i \eta - \beta)^3},
\end{equation}
this means that 
\begin{equation*}
	\widehat{T_0\, h}:= m_0(\xi, \eta, \tau) \widehat{h},
\end{equation*}
where, in order to simplify the notation, we use\, $\widehat{ }\,$ to denote the Fourier transform in $S'(\R^3).$

The proof of Lemma~\ref{lemma2.1} is based on two previous lemmas, these lemmas express respectively the boundedness of the operator $T_0$ and $(\partial_x - \lambda)^k (\partial_y - \beta)^l T_0$ where $k,l$ are non negative integers with $0\leq k+l\leq 2$ (actually we need just the decoupled options, that is $(k,l)=(0,0), (1,0), (0,1), (2,0)$ and $(0,2)$).
 
\begin{lemma}\label{boundedness_T0}
	Let $h\in L^1(\R^3)$ with $\norm{h}_{L_t^1 L_{x y}^2}(\R^3)<\infty.$ Then for all $(\lambda, \beta)\neq (0,0),$ $m_0 \widehat{h} \in S'(\R^3)$ and $[m_0 \widehat{h}]\invf{}$ defines a bounded function from $\R_t$ with values in $L_{x y}^2.$ Besides,
	\begin{equation}\label{boundedness}
		\norm{[m_0 \widehat{h}]\invf{}\,(t)}_{L_{x y}^2(\R^2)}\leq \norm{h}_{L_t^1 L_{x y}^2(\R^3)} \quad \forall t \in \R,
	\end{equation}
	where $\invf{}$ denotes the inverse Fourier transform in $S'(\R^3)$.
	\begin{remark}
		Clearly the previous inequality gives the boundedness of the operator $T_0$ indeed, by its definition, from~\eqref{boundedness} follows that
		\begin{equation*}
			\norm{[T_0 h](t)}_{L_{x y}^2(\R^2)}\leq \norm{h}_{L_t^1 L_{x y}^2(\R^3)}  \quad \forall t \in \R.
		\end{equation*}
	\end{remark}
		\begin{proof}
			First of all we want to write the symbol $m_0(\xi, \eta, \tau)$ in a more useful way, precisely it is not difficult to see that the following holds:
			\begin{equation*}
				m_0(\xi,\eta,\tau)=\frac{-i}{\tau + a(\xi,\eta) + i b(\xi,\eta)},
			\end{equation*}
			where 
			\begin{equation*}
				 a(\xi,\eta)=-\xi^3 +3\xi \lambda^2 -\eta^3 + 3 \eta \beta^2 \qquad \text{and} \qquad
				 b(\xi,\eta)= \lambda^3 -3\xi^2\lambda +\beta^3 -3\eta^2\beta.
			\end{equation*}
			Before going any further we want to recall some useful properties of the Fourier transform.
			\begin{remark}\label{anti_f}
				Our definition for the $1$-dimensional Fourier transform is
				\begin{equation}\label{Fourier1d}
					\widehat{f}(\tau)= \frac{1}{\sqrt{2\pi}} \int_{\R} e^{-i\tau t} f(t)\, dt.
				\end{equation}
				Making a straightforward computation it is not difficult to see that, defining
				\begin{equation*}
					g(\tau)=\frac{-i}{\tau + i b}, \qquad b\neq 0,
				\end{equation*}
				the inverse Fourier transform of $g$ has this form
				\begin{equation}\label{inverse_g}
					\widecheck{g}(t)=
														\begin{system}
															\phantom{-}\sqrt{2\pi}\, \chi_{(0, + \infty)}(t) e^{tb} \quad b<0, \vspace{+0.3cm} \\
															-\sqrt{2\pi}\, \chi_{(- \infty,0)}(t) e^{tb} \quad b>0,
													  \end{system}
				\end{equation}
				where, as usual, for a set $A,$ $\chi_A$ denotes the characteristic function of $A.$
				
				Considering the translation by the real number $a$ of $g,$ that is defining $G(\tau)=g(\tau + a),$ from~\eqref{inverse_g} and the property that the translation in the moment space is a multiplication by a phase factor in the position space and vice-versa, in other words
				\begin{equation*}
					\widecheck{g}(\cdot + a)(t) = e^{-ita} \widecheck{g}(t),
				\end{equation*}
				one has
				\begin{equation*}
					\widecheck{G}(t)=
																				\begin{system}	
																							\phantom{-}\sqrt{2\pi}\, \chi_{(0, + \infty)}(t) e^{tb} e^{-ita} \quad b<0, \vspace{+0.3cm} \\
															-\sqrt{2\pi}\, \chi_{(- \infty,0)}(t) e^{tb} e^{-ita} \quad b>0.																	
																				\end{system}
				\end{equation*}
			\end{remark}
			
			With the previous remark in mind we can say that for a fixed pair $(\xi, \eta)$ with $b(\xi, \eta)\neq 0$ and $t\in \R$ we have
			\begin{equation*}
					[m_0(\xi, \eta, \cdot_\tau)]\invt{}(t)=
															\begin{system}
							\phantom{-}\sqrt{2\pi}\, \chi_{(0, + \infty)}(t) e^{tb(\xi,\eta)} e^{-ita(\xi, \eta)} \quad b(\xi,\eta)<0, \vspace{+0.3cm} \\
															-\sqrt{2\pi}\, \chi_{(- \infty,0)}(t) e^{tb(\xi,\eta)} e^{-ita(\xi,\eta)} \quad b(\xi,\eta)>0.
															\end{system}
			\end{equation*}
			Clearly the magnitude of the right-hand side is bounded by $\sqrt{2\pi}.$
			
			Now we need to compute the quantity $\big[ m_0(\xi, \eta, \cdot_\tau) \widehat{h}(\xi, \eta, \cdot_\tau)\big]\invt{}(t).$
			
			In order to do that we recall that under our definition of the Fourier transform~\eqref{Fourier1d} and its inverse, the following property holds:
			\begin{equation*}
				\widecheck{f g}(t)=\frac{\widecheck{f}(t) \ast \widecheck{g}(t)}{\sqrt{2\pi}},
			\end{equation*}
			moreover using that $\widehat{h}=h\fxy{} \ft{},$ one easily obtains
			\begin{equation*}
				\begin{split}
					\big[ m_0(\xi, \eta,\cdot_\tau) \widehat{h}(\xi, \eta, \cdot_\tau)\big]\invt{}(t)
					  &= \frac{[m_0(\xi, \eta, \cdot_\tau)]\invt{}(t) \ast h(\cdot_x, \cdot_y, t)\fxy{}(\xi,\eta)}{\sqrt{2\pi}}\\
																			 &=\frac{1}{\sqrt{2\pi}} \int_{\R_s} [m_0(\xi, \eta, \cdot_\tau)]\invt{}(t-s) h(\cdot_x, \cdot_y, s)\fxy{}(\xi,\eta) \, ds\\
																			 &=
																				\begin{system}
																				\ \ \ \int_{\R_s} \chi_{(0, + \infty)}(t-s) e^{(t-s) b(\xi,\eta)} e^{-i (t-s)a(\xi,\eta)} h(\cdot_x, \cdot_y, s)\fxy{}(\xi,\eta) \, ds \quad b(\xi,\eta)<0, \vspace{+0.3cm} \\			
																				- \int_{\R_s} \chi_{(- \infty,0)}(t-s) e^{(t-s) b(\xi,\eta)} e^{-i (t-s)a(\xi,\eta)} h(\cdot_x, \cdot_y, s)\fxy{}(\xi,\eta) \, ds \quad b(\xi,\eta)>0.
																				\end{system}
				\end{split}
			\end{equation*}
			Let us observe that for $(\lambda, \beta)\neq (0,0)$ since the set $\{(\xi, \eta)\colon b(\xi,\eta)=0\}$ represents an ellipse, it has measure zero in $\R^2,$ this gives, by applying Plancherel's formula and Minkowski's integral inequality, that for all $t\in \R$
			\begin{equation*}
				\begin{split}
					\norm{[m_0 \widehat{h}]\invf{}\,(\cdot_x, \cdot_y, t)}_{L_{x y}^2(\R^2)}&=\norm{[m_0 \widehat{h}] \invt{}(\cdot_\xi, \cdot_\eta, t)}_{L_{\xi \eta}^2(\R^2)} \leq \int_{\R_s} \norm{h\fxy{}(\cdot_\xi, \cdot_\eta, s)}_{L_{\xi \eta}^2(\R^2)}\\
															&= \norm{h(\cdot_x, \cdot_y, \cdot_t)}_{L_t^1 L_{x y}^2(\R^3)}<\infty. 
				\end{split}
			\end{equation*}
		\end{proof}
\end{lemma}
As previously anticipated, we are going to prove the boundedness of the operator $(\partial_x - \lambda)^k (\partial_y - \beta)^l T_0.$ precisely, we will prove the following lemma
\begin{lemma}\label{boundedness_der_t0}
	Let $h \in L^1(\R^3)$ with $\norm{h}_{L_x^1 L_{y t}^2}(\R^3)<\infty.$ For $\beta\geq 1,$ $\lambda\geq 1,$ $k,l \in \{0,1,2\},$ and $0\leq k+l\leq 2,$ let
	\begin{equation*}
		m_{k,l}(\xi, \eta, \tau):= (i\xi-\lambda)^k (i\eta -\beta)^l m_0(\xi,\eta,\tau),
	\end{equation*}
	with $m_0$ as in~\eqref{symbol}, the symbol associated with the operator $(\partial_x - \lambda)^k (\partial_y - \beta)^l\, T_0.$ Then $m_{k,l} \widehat{h} \in S'(\R^3)$ and 
	\begin{equation*}
		\norm{[m_{k,l} \widehat{h}]\invf\, (t)}_{L_x^\infty L_{t y}^2(\R^3)}\leq \norm{h}_{L_x^1 L_{y t}^2(\R^3)}.
	\end{equation*}
	\begin{remark}
	As in Lemma~\ref{boundedness_T0}, from the previous inequality we can conclude the boundedness of the operator $(\partial_x -\lambda)^k (\partial_y - \beta)^l T_0,$ indeed the as a trivial consequence we have
	\begin{equation*}
		\norm{[(\partial_x - \lambda)^k (\partial_y - \beta)^l T_0] h}_{L_x^\infty L_{t y}^2(\R^3)}\leq C \norm{h}_{L_x^1 L_{y t}^2(\R^3)}.
	\end{equation*}
	\end{remark}
 	\begin{proof}
		We will only consider the case $k=2$ and $l=0.$ Since the proofs of other cases are similar, for brevity, we will omit them. First of all let us note that 
		\begin{equation*}
			m_{2,0}(\xi, \eta, \tau)= \frac{-i (\xi +i \lambda)^2}{ [(\xi + i \lambda)^3 + (\eta + i \beta)^3 - \tau]}.
		\end{equation*}
		Defining $v:= \xi+ i \lambda$ and $w:= \eta + i \beta$ we can re-write the preceding as
		\begin{equation*}
			m_{2,0}(\xi, \eta, \tau)= \frac{-i v^2}{v^3 + w^3 - \tau}.
		\end{equation*}
     The polynomial $P(v):=v^3 + w^3 - \tau$ has got, as a multiple root, just $v=0,$ but since under our hypothesis $v$ is always different from zero, we can assume $P(v)$ not to have multiple roots. This allows us to use the following decomposition in partial fractions
						\begin{equation*}
							m_{2,0}= \sum_{j=1}^3 \frac{-i v_j^2}{3(v-v_j) v_j^2}= \sum_{j=1}^3 \frac{-i}{3(\xi - \Re(v_j) + i[\lambda - \Im(v_j)])}= \frac{1}{3} \sum_{j=1}^3 \frac{-i}{\xi + a_j(\eta, \tau) + i b_j(\eta, \tau)}
						\end{equation*}
						where $v_j,$ $j=1,2,3$ are the different roots of $P, a_j(\eta, \tau)= -\Re(v_j)$ and $b_j(\eta, \tau)= \lambda - \Im(v_j).$
						Moving on as in Lemma~\ref{boundedness_T0}, that is using the Remark~\ref{anti_f},  for a fixed pair $(\eta, \tau)$ such that $b(\eta, \tau)\neq 0,$ making use of the linearity of the inverse Fourier transform we have
						\begin{equation*}
							[m_{2,0}(\cdot_\xi, \eta, \tau)]\invx{} \, (x)= \begin{system}
				 \phantom{-}\frac{1}{3}\sum_{j=1}^3 \sqrt{2\pi}\, \chi_{(0, + \infty)}(x) e^{x b_j(\eta, \tau)} e^{-ix a_j(\eta, \tau)} \quad b_j(\eta, \tau)<0, \vspace{+0.3cm} \\
															-\frac{1}{3}\sum_{j=1}^3 \sqrt{2\pi}\, \chi_{(- \infty,0)}(x) e^{xb_j(\eta, \tau)} e^{-ixa_j(\eta, \tau)} \quad b_j(\eta, \tau)>0. 				
				\end{system}
						\end{equation*}
						Clearly the magnitude of the right-hand side is bounded by $\sqrt{2\pi}.$
						
						Let us observe that the set $\{(\eta, \tau)\colon \Im(v_j) - \lambda=0\}$ has two-dimensional measure zero. Therefore using similar computations to those performed in Lemma~\ref{boundedness_T0} we get that for all $x \in \R$
						\begin{equation*}
							\begin{split}
								\norm{[m_{2,0}\widehat{h}] \invf{}\, (x, \cdot_y, \cdot_t)}_{L_{y t}^2(\R^2)}&=\norm{[m_{2,0}\widehat{h}] \invx{}\, (x, \cdot_\eta, \cdot_\tau)}_{L_{\eta \tau}^2(\R^2)}\leq \int_{\R_z} \norm{h\fyt{}\,(z, \cdot_\eta, \cdot_\tau)}_{L_{\eta \tau}^2(\R^2)}\\
																				 &=\norm{h(\cdot_x, \cdot_y, \cdot_t)}_{L_x^1 L_{y t}^2(\R^3)}< \infty.
							\end{split}
						\end{equation*}
	\end{proof}
\end{lemma}
Now we are in position to prove Lemma~\ref{lemma2.1}. Even if the proof of this lemma is similar to the one for the corresponding result in~\cite{B_I_M}, we will provide it for sake of completeness.  
\begin{proof}[Proof of Lemma~\ref{lemma2.1}] The proof of estimate~\eqref{no_deriv} follows from Lemma~\ref{boundedness_T0} and the proof of~\eqref{deriv} follows from Lemma~\ref{boundedness_der_t0}. We only prove the estimate~\eqref{deriv} for $L=\partial_x^2.$

For $\varepsilon \in (0, \frac{1}{4})$ let $\eta_\varepsilon$ be a function in $C_0^\infty(\R)$ of the time variable $t$ such that $\eta_\varepsilon(t)=1$ if $t\in [2\varepsilon, 1- 2\varepsilon],$ $\supp \eta_\varepsilon \subset [\varepsilon, 1- \varepsilon],$ $\eta_\varepsilon$ increasing in $[\varepsilon, 2\varepsilon]$ and decreasing in $[1-2\varepsilon, 1 - \varepsilon].$ 
Let us define for all $t \in \R$
\begin{equation*}
	w_\varepsilon(t):= \eta_\varepsilon(t) w(t),
\end{equation*} 
where with an abuse of notation $w$ represents the extension of $w$ which is identically zero outside $[0,1].$ We define
\begin{equation*}
	h_\varepsilon:= e^{\lambda x} e^{\beta y} (\partial_t + \partial_x^3 + \partial_y^3) w_\varepsilon,
\end{equation*}
then, more explicitly 
\begin{equation}\label{h_eps}
	h_\varepsilon= \eta_\varepsilon' e^{\lambda x} e^{\beta y} w + h_0,
\end{equation}
where
\begin{equation*}
	h_0:= \eta_\varepsilon e^{\lambda x} e^{\beta y} (\partial_t + \partial_x^3 + \partial_y^3) w.
\end{equation*}
It is not difficult to see that $h_\varepsilon$ can be re-written as
\begin{equation*}
	h_\varepsilon= [e^{\lambda x} e^{\beta y} (\partial_t + \partial_x^3 + \partial_y^3) e^{-\lambda x} e^{-\beta y}] e^{\lambda x} e^{\beta y} w_\varepsilon= H_{\lambda, \beta} (e^{\lambda x}e^{\beta y} w_\varepsilon).
\end{equation*}
This means that 
\begin{equation*}
	e^{\lambda x} e^{\beta y} w_\varepsilon=T_0 h_\varepsilon=[m_0 \widehat{h_\varepsilon}]\invf{}.
\end{equation*}
Now we consider $e^{\lambda x} e^{\beta y}\, \partial_x^2 w_\varepsilon.$ It is easy to see that
\begin{equation*}
	e^{\lambda x} e^{\beta y} \partial_x^2 w_\varepsilon=(e^{\lambda x} e^{\beta y}\, \partial_x^2\, e^{-\lambda x} e^{- \beta y}) e^{\lambda x} e^{\beta y} w_\varepsilon=(\partial_x - \lambda)^2 e^{\lambda x} e^{\beta y} w_{\varepsilon}=(\partial_x - \lambda)^2 T_0 h_\varepsilon= [m_{2,0} \widehat{h_\varepsilon}] \invf{}.
\end{equation*}
From the previous identity and~\eqref{h_eps}, one gets
\begin{equation}\label{first_estimate}
	\begin{split}
		\norm{e^{\lambda x} e^{\beta y} \partial_x^2 w_\varepsilon}_{L_x^\infty L_{y t}^2}&=\norm{[m_{2,0} \widehat{h_\varepsilon}]\invf{}\,\,}_{L_x^\infty L_{y t}^2}\\
																	 &\leq \norm{\chi_{[0,1]}(\cdot_t)[m_{2,0} (\eta_\varepsilon' e^{\lambda x} e^{\beta y} w)\dirf{}\,\,]\invf{}\,\,}_{L_x^\infty L_{y t}^2} + \norm{[m_{2,0} \widehat{h_0}]\invf{}\,\,}_{L_x^\infty L_{y t}^2}. 
 \end{split}																		
\end{equation}
First of all let us consider the second term on the right-hand side, using the hypotheses of Lemma~\ref{lemma2.1} we can apply Lemma~\ref{boundedness_der_t0} to $h_0,$ this gives
\begin{equation}\label{second_term}
	\norm{[m_{2,0} \widehat{h_0}]\invf{}\,\,}_{L_x^\infty L_{y t}^2}\leq \norm{h_0}_{L_x^1 L_{y t}^2}. 
\end{equation}
Now we need to provide an estimate for the first term on the right-hand side of~\eqref{first_estimate}. Using our definition of $m_{k, l}(\xi,\eta, \tau)$ we get
\begin{equation*}
	\begin{split}
		\norm{\chi_{[0,1]}(\cdot_t)[m_{2,0} (\eta_\varepsilon' e^{\lambda x} e^{\beta y} w)\dirf{}\,\,]\invf{}\,\,}_{L_x^\infty L_{y t}^2} &= \norm{\chi_{[0,1]}(\cdot_t)[-(\xi + i \lambda)^2 m_0 (\eta_\varepsilon' e^{\lambda x} e^{\beta y} w)\dirf{}\,\,]\invf{}\,\,}_{L_x^\infty L_{y t}^2}\\
																		&= \norm{\chi_{[0,1]}(\cdot_t)[m_0 \widehat{g}] \invf{}\,\,}_{L_x^\infty L_{y t}^2},
	\end{split}
\end{equation*}
where $\widehat{g}= -(\xi + i \lambda)^2 (\eta_\varepsilon' e^{\lambda x} e^{\beta y} w)\dirf{}\,\,.$

For a fixed pair $(y, t) \in \R^2$ one has
\begin{equation*}
	\begin{split}
		\norm{\chi_{[0,1]}(t)[m_0 \widehat{g}]\invf{}\,(\cdot_x, y,t)}_{H_x^1}&= \norm{(1 + (\cdot_\xi)^2)^\frac{1}{2} \chi_{[0,1]}(t)[m_0 \widehat{g}] \invet{}(\cdot_\xi, y,t)}_{L_\xi^2}\\
					&=\norm{(1 + (\cdot_\xi)^2)^\frac{1}{2} (\cdot_\xi + i\lambda)^2 \chi_{[0,1]}(t) [m_0 (\eta_\varepsilon' e^{\lambda x} e^{\beta y} w)\dirf{}\,\,]\invet{}(\cdot_\xi, y,t)}_{L_\xi^2}.
	\end{split}
\end{equation*}
Since
\begin{equation*}
	(1 + \xi^2)^\frac{1}{2} \abs{\xi + i\lambda}^2 \leq (1+ \xi^2)^\frac{1}{2} (1 + \xi^2 + \lambda^2) \leq (1 + \xi^2)^\frac{3}{2} (1+ \lambda^2)
\end{equation*}
we obtain
\begin{equation*}
	\norm{\chi_{[0,1]}(t)[m_0 \widehat{g}]\invf{}\,(\cdot_x, y,t)}_{H_x^1}\leq (1+ \lambda^2)\norm{J_x^3\,\chi_{[0,1]}(t)[m_0 (\eta_\varepsilon' e^{\lambda x} e^{\beta y} w)\dirf{}\,\,]\invf{}\,\,(\cdot_x, y,t)}_{L_x^2}.
\end{equation*}
\begin{remark}
	We emphasize that here $J_x^3$ denotes the operator defined through the Fourier transform just in the $x$ variable by
	\begin{equation*}
		\widehat{J_x^3g}(\xi):= (1 + \xi^2)^\frac{3}{2} \widehat{g}(\xi).
	\end{equation*}
\end{remark}
Now, using that $H_x^1(\R) \hookrightarrow L_x^\infty(\R)$ we have
\begin{equation*}
	\begin{split}
		\abs{\chi_{[0,1]}(t)[m_0 \widehat{g}]\invf{}\,\,(x,y,t)}&\leq c \norm{\chi_{[0,1]}(t)[m_0 \widehat{g}]\invf{}\,(\cdot_x, y,t)}_{H_x^1}\\
						&\leq c(1+ \lambda^2)\norm{J_x^3\chi_{[0,1]}(t)[m_0 (\eta_\varepsilon' e^{\lambda x} e^{\beta y} w)\dirf{}\,\,]\invf{}\,\,(\cdot_x, y,t)}_{L_x^2}.
	\end{split}
\end{equation*}
Therefore, for $x\in \R,$ by virtue of Lemma~\ref{boundedness_T0} one obtains
\begin{equation}\label{first_term}
	\begin{split}
		\norm{ \chi_{[0,1]}(\cdot_t) [m_0 \widehat{g}]\invf{}\,\,(x, \cdot_y, \cdot_t)}_{L_{y t}^2}&\leq c (1+ \lambda^2)\norm{J_x^3\chi_{[0,1]}(\cdot_t)[m_0 (\eta_\varepsilon' e^{\lambda x} e^{\beta y} w)\dirf{}\,\,]\invf{}\,\,}_{L^2}\\
			&\leq c (1+ \lambda^2)\norm{J_x^3 [m_0 (\eta_\varepsilon' e^{\lambda x} e^{\beta y} w)\dirf{}\,\,]\invf{}\,\,}_{L_t^\infty L_{x y}^2}\\
			&\leq c (1 + \lambda^2) \norm{(1 + (\cdot_\xi)^2 + (\cdot_\eta)^2)^\frac{3}{2}[m_0(\eta_\varepsilon' e^{\lambda x} e^{\beta y} w)\dirf{}\,\,]\invt{}}_{L_t^\infty L_{\xi \eta}^2}\\
			&= c (1 + \lambda^2) \norm{[m_0(\eta_\varepsilon' J^3(e^{\lambda x} e^{\beta y} w))\dirf{}\,\,]\invf{}\,\,}_{L_t^\infty L_{x y}^2}\\
			&\leq  c (1 + \lambda^2) \norm{\eta_\varepsilon' J^3(e^{\lambda x} e^{\beta y} w)}_{L_t^1 L_{x y}^2}.
	\end{split}
\end{equation}
Now plugging~\eqref{second_term} and~\eqref{first_term} in~\eqref{first_estimate} and using the explicit definition of $h_0,$ it follows that
\begin{equation}\label{final}
	\norm{e^{\lambda x} e^{\beta y} \partial_x^2 w_\varepsilon}_{L_x^\infty L_{y t}^2}\leq c (1 + \lambda^2) \norm{\eta_\varepsilon' J^3(e^{\lambda x} e^{\beta y} w)}_{L_t^1 L_{x y}^2} + \norm{\eta_\varepsilon e^{\lambda x} e^{\beta y}(\partial_t + \partial_x^3 + \partial_y^3) w}_{L_x^1 L_{y t}^2}.
\end{equation}
First of all we want to prove that the left-hand side of~\eqref{final} goes to $\norm{e^{\lambda x} e^{\beta y} \partial_x^2 w}_{L_x^\infty L_{y t}^2}$ as $\varepsilon$ tends to $0^+.$ Since by our hypotheses we are assuming $w(t)$ to be compactly supported, without loss of generality we may suppose $\supp w(t)\subset [-M, M] \times [-M, M]$ for all $t\in [0,1].$ Making use of the fact that $\partial_x^2 w(t)\in H^2(\R^2) \hookrightarrow L^\infty(\R^2),$ we get
\begin{equation*}
	\begin{split}
		\norm{e^{\lambda x} e^{\beta y} \partial_x^2 w_\varepsilon - e^{\lambda x} e^{\beta y} \partial_x^2 w }_{L_x^\infty L_{y t}^2}
		&= \essupp_{x\in [-M, M]}\Big[\int_0^1 \int_{-M}^M e^{2\lambda x} e^{2\beta y} (\eta_\varepsilon(t) -1)^2 (\partial_x^2 w)^2(x,y,t) \,dy\, dt \Big]^\frac{1}{2}\\
														&\leq c e^{\lambda M} e^{\beta M} \norm{\partial_x^2 w}_{C([0,1]; H^2(\R^2))} (2M)^\frac{1}{2} \Big[ \int_0^{2\varepsilon} dt + \int_{1-2\varepsilon}^1 dt \Big]^\frac{1}{2} \xrightarrow[]{\varepsilon \to 0^+} 0.
	\end{split}
\end{equation*}
With respect to the first term of the right-hand side of~\eqref{final} we can show that
\begin{equation*}
	\begin{split}
		\norm{\eta_\varepsilon' J^3 (e^{\lambda x} e^{\beta y} w)}_{L_t^1 L_{x y}^2}&= \int_0^1 \abs{\eta_\varepsilon'(t)} \norm{J^3(e^{\lambda x} e^{\beta y} w(t))}_{L_{x y}^2}\, dt\\
															&=\int_\varepsilon^{2\varepsilon} \eta_\varepsilon'(t) \norm{J^3(e^{\lambda x} e^{\beta y} w(t))}_{L_{x y}^2}\, dt - \int_{1-2\varepsilon}^{1-\varepsilon} \eta_\varepsilon'(t) \norm{J^3(e^{\lambda x} e^{\beta y} w(t))}_{L_{x y}^2} \, dt\\
															&=\int_\varepsilon^{2\varepsilon} \eta_\varepsilon'(t) \big( \norm{J^3(e^{\lambda x} e^{\beta y} w(t))}_{L_{x y}^2} - \norm{ J^3(e^{\lambda x} e^{\beta y} w(0))}_{L_{x y}^2} \big)\,dt\\
															&\quad + \norm{J^3(e^{\lambda x} e^{\beta y} w(0))}_{L_{x y}^2}\\
															&\quad - \int_{1-2\varepsilon}^{1-\varepsilon} \eta_\varepsilon'(t) \big( \norm{J^3(e^{\lambda x} e^{\beta y} w(t))}_{L_{x y}^2} - \norm{J^3(e^{\lambda x} e^{\beta y} w(1))}_{L_{x y}^2} \big)\, dt\\
															&\quad + \norm{J^3(e^{\lambda x} e^{\beta y} w(1))}_{L_{x y}^2},
	\end{split}
\end{equation*}
since $e^{\lambda x} e^{\beta y} w\in C([0,1]; H^3(\R^2)),$ it is easy to see that
\begin{equation*}
	\norm{\eta_\varepsilon' J^3 (e^{\lambda x} e^{\beta y} w)}_{L_t^1 L_{x y}^2} \xrightarrow[]{\varepsilon \to 0^+} \norm{J^3(e^{\lambda x} e^{\beta y} w(0))}_{L_{x y}^2} + \norm{J^3(e^{\lambda x} e^{\beta y} w(1))}_{L_{x y}^2}.
\end{equation*}
Now only the estimate of the second term of the right-hand side of~\eqref{final} is missing. Taking into account that $\supp w \subset [-M, M] \times [-M,M] \times [0,1]$ and using the dominated convergence theorem we can conclude that
\begin{equation*}
	\norm{(\eta_\varepsilon -1) e^{\lambda x} e^{\beta y} (\partial_t + \partial_x^3 + \partial_y^3)w}_{L_x^1 L_{y t}^2}\leq (2M)^\frac{1}{2} e^{\lambda M} e^{\beta M} \norm{(\eta_\varepsilon -1)(\partial_t + \partial_x^3 + \partial_y^3)w}_{L^2}\xrightarrow[]{\varepsilon \to 0^+}0. 
\end{equation*}
Putting all these estimates together and using $\beta \geq 1$ we obtain
\begin{equation}\label{kl1}
	\begin{split}
	\norm{e^{\lambda x} e^{\beta y} \partial_x^2 w}_{L_x^\infty L_{y t}^2}\leq& c (\lambda^2 + \beta^2) \big(\norm{J^3(e^{\lambda x} e^{\beta y} w(0))}_{L^2} + \norm{J^3(e^{\lambda x} e^{\beta y} w(1))}_{L^2} \big)\\
										 &+ \norm{e^{\lambda x} e^{\beta y} (\partial_t + \partial_x^3 + \partial_y^3)w}_{L_x^1 L_{y t}^2}.
	\end{split}
\end{equation}
In order to conclude the proof we need the following remark.

	An equivalent way to write the estimate~\eqref{deriv} is the following
	\begin{equation*}
		\begin{split}
			\norm{e^{j \lambda x} e^{k \beta y} \partial_x^2 w}_{L_x^\infty L_{y t}^2}\leq& c (\lambda^2 + \beta^2) \big(\norm{J^3(e^{j \lambda x} e^{k \beta y} w(0))}_{L^2} + \norm{J^3(e^{j \lambda x} e^{k\beta y} w(1))}_{L^2} \big)\\
										 &+ \norm{e^{j \lambda x} e^{k \beta y} (\partial_t + \partial_x^3 + \partial_y^3)w}_{L_x^1 L_{y t}^2},
		\end{split}
	\end{equation*}
	for $j \in \{-1,1\}$ and $k\in \{-1, 1\}.$
	
	We have already proved the former estimate for $j=k=1.$ Our aim is to show that the other cases follow in a similar way and so omit them. 
	
	The first step we have to perform is to modify the definition of the multipliers $m_0$ and $m_{k, l}$ considering, instead of $(i \xi - \lambda)$ and $(i \eta - \beta),$ the other three possible pairs: $(i \xi + \lambda)$ and $(i\eta + \beta)$ if we want to estimate $\norm{e^{-\lambda x} e^{-\beta y} Lw}_{L_x^\infty L_{y t}^2},$ $(i\xi + \lambda)$ and $(i\eta - \beta)$ if we want to estimate $\norm{e^{-\lambda x} e^{\beta y} Lw}_{L_x^\infty L_{y t}^2},$ $(i \xi - \lambda)$ and $(i\eta + \beta)$ for the estimate of $\norm{e^{\lambda x} e^{-\beta y} Lw}_{L_x^\infty L_{y t}^2}.$
	
	Since in order to prove~\eqref{kl1} we strongly used the estimates in Lemma~\ref{boundedness_T0} and~\ref{boundedness_der_t0}, we would like them to hold also for the modified versions of $m_0$ and $m_{k,l}$ written above.  
	But one can easily see that this is true just revisiting the proof of the two lemmas with the new definitions of $m_0$ and $m_{k,l}.$
		This concludes the proof of our lemma.
\end{proof}

Now we shall extend the result in Lemma~\ref{lemma2.1} to operators as in~\eqref{operator_difference}, namely we prove Lemma~\ref{lemma2.2}.
\begin{proof}[Proof of Lemma~\ref{lemma2.2}]
		From Lemma~\ref{lemma2.1} and using the fact that $\norm{\cdot}_{L^2(\R^2 \times[0,1])} \leq \norm{\cdot}_{L_t^\infty L_{x y}^2(\R^2 \times [0,1])},$ it follows that
		\begin{multline}\label{first_complete}
				\norm{e^{\lambda \abs{x}} e^{\beta \abs{y}} w}_{L^2} \leq \norm{e^{\lambda \abs{x}} e^{\beta \abs{y}} w(0)}_{L^2} + \norm{e^{\lambda \abs{x}} e^{\beta \abs{y}} w(1)}_{L^2}\\
									 + \norm{e^{\lambda \abs{x}}  e^{\beta \abs{y}} (\partial_t + \partial_x^3 + \partial_y^3 + a_1(\partial_x + \partial_y) + a_0)w}_{L_t^1 L_{x y}^2} 
									+ \norm{e^{\lambda \abs{x}} e^{\beta \abs{y}} (a_1(\partial_x + \partial_y) + a_0)w}_{L_t^1 L_{x y}^2},
		\end{multline}
		and
		\begin{multline}\label{second_complete}
				\norm{e^{\lambda \abs{x}} e^{\beta \abs{y}} Lw}_{L_x^\infty\, L_{y t}^2} 
				\leq c\, (\lambda^2 + \beta^2)\big (\norm{J^3(e^{\lambda \abs{x}} e^{\beta \abs{y}} w(0))}_{L^2} + \norm{J^3(e^{\lambda \abs{x}} e^{\beta \abs{y}} w(1))}_{L^2} \big) \\
				+ \norm{e^{\lambda \abs{x}} e^{\beta \abs{y}} (\partial_t + \partial_x^3 + \partial_y^3 + a_1(\partial_x + \partial_y) + a_0)w}_{L_x^1 L_{y t}^2}\\ 
									+ \norm{e^{\lambda \abs{x}} e^{\beta \abs{y}} (a_1(\partial_x + \partial_y) + a_0)w}_{L_x^1 L_{y t}^2}.
	\end{multline}
		We are interested in considering the last terms in the former estimates.
		
		We first see $\norm{e^{\lambda \abs{x}} e^{\beta \abs{y}} (a_1(\partial_x + \partial_y) + a_0)w}_{L_t^1 L_{x y}^2}$ using that $\norm{\cdot}_{L_t^1 L_{x y}^2(\R^2 \times [0,1])} \leq \norm{\cdot}_{L^2(\R^2\times [0,1])},$ we easily obtain
		\begin{equation*}
			\begin{split}
				\norm{e^{\lambda \abs{x}} e^{\beta \abs{y}} (a_1(\partial_x + \partial_y) + a_0)w}_{L_t^1 L_{x y}^2}
				 &\leq \norm{e^{\lambda \abs{x}} e^{\beta \abs{y}} (a_1(\partial_x + \partial_y) + a_0)w}_{L^2}\\
			 &\leq \norm{a_1}_{L_x^2 L_{y t}^\infty} \norm{e^{\lambda \abs{x}} e^{\beta \abs{y}} (\partial_x + \partial_y) w}_{L_x^\infty L_{y t}^2} + \norm{a_0}_{L^\infty} \norm{e^{\lambda \abs{x}} e^{\beta \abs{y}}w}_{L^2}.
			\end{split}
		\end{equation*}
		Let us consider now $\norm{e^{\lambda \abs{x}} e^{\beta \abs{y}} (a_1(\partial_x + \partial_y) + a_0)w}_{L_x^1 L_{y t}^2},$ making use of the H\"older's inequality, one gets
		\begin{equation*}
			\norm{e^{\lambda \abs{x}} e^{\beta \abs{y}} (a_1(\partial_x + \partial_y) + a_0)w}_{L_x^1 L_{y t}^2}\leq \norm{a_1}_{L_x^1 L_{y t}^\infty} \norm{e^{\lambda \abs{x}} e^{\beta \abs{y}}(\partial_x + \partial_y)w}_{L_x^\infty L_{y t}^2} + \norm{a_0}_{L_x^2 L_{y t}^\infty} \norm{e^{\lambda \abs{x}} e^{\beta \abs{y}} w}_{L^2}.
		 \end{equation*}
		Plugging the previous estimates into~\eqref{first_complete} and~\eqref{second_complete} and summing them together we have
		\begin{multline*}
				\norm{e^{\lambda \abs{x}} e^{\beta \abs{y}} w}_{L^2} + \sum_{0< k+ l \leq 2} \norm{e^{\lambda \abs{x}} e^{\beta \abs{y}} \partial_x^k \partial_y^l w}_{L_x^\infty\, L_{y t}^2}
			\leq c\, (\lambda^2 + \beta^2)\big (\norm{J^3(e^{\lambda \abs{x}} e^{\beta \abs{y}} w(0))}_{L^2} + \norm{J^3(e^{\lambda \abs{x}} e^{\beta \abs{y}} w(1))}_{L^2} \big)\\
									 + \norm{e^{\lambda \abs{x}} e^{\beta \abs{y}} (\partial_t + \partial_x^3 + \partial_y^3 + a_1(\partial_x + \partial_y) + a_0)w}_{L_t^1 L_{x y}^2 \cap L_x^1 L_{y t}^2}\\
									 + \norm{a_0}_{L^\infty \cap L_x^2 L_{y t}^\infty} \norm{e^{\lambda \abs{x}} e^{\beta \abs{y}} w}_{L^2}
									+ \norm{a_1}_{L_x^2 L_{y t}^\infty \cap L_x^1 L_{y t}^\infty} \norm{e^{\lambda \abs{x}} e^{\beta \abs{y}} (\partial_x + \partial_y) w}_{L_x^\infty L_{y t}^2}.					
		\end{multline*}
		Under our hypotheses about $a_0$ and $a_1$ we have
		\begin{multline}\label{hiding_norm}
			\norm{e^{\lambda \abs{x}} e^{\beta \abs{y}} w}_{L^2} + \sum_{0< k+ l \leq 2} \norm{e^{\lambda \abs{x}} e^{\beta \abs{y}} \partial_x^k \partial_y^l w}_{L_x^\infty\, L_{y t}^2}
			\leq c\, (\lambda^2 + \beta^2)\big (\norm{J^3(e^{\lambda \abs{x}} e^{\beta \abs{y}} w(0))}_{L^2} + \norm{J^3(e^{\lambda \abs{x}} e^{\beta \abs{y}} w(1))}_{L^2} \big)\\
									 + \norm{e^{\lambda \abs{x}} e^{\beta \abs{y}} (\partial_t + \partial_x^3 + \partial_y^3 + a_1(\partial_x + \partial_y) + a_0)w}_{L_t^1 L_{x y}^2 \cap L_x^1 L_{y t}^2}\\
									+ \frac{1}{2} \Big( \norm{e^{\lambda \abs{x}} e^{\beta \abs{y}} w}_{L^2} + \sum_{0< k+ l \leq 2} \norm{e^{\lambda \abs{x}} e^{\beta \abs{y}} \partial_x^k \partial_y^l w}_{L_x^\infty\, L_{y t}^2} \Big).
		\end{multline}
		Hence, absorbing the last term in the left-hand side, we have
		\begin{multline*}
			\norm{e^{\lambda \abs{x}} e^{\beta \abs{y}} w}_{L^2} + \sum_{0< k+ l \leq 2} \norm{e^{\lambda \abs{x}} e^{\beta \abs{y}} \partial_x^k \partial_y^l w}_{L_x^\infty\, L_{y t}^2}\\
			\leq c\, (\lambda^2 + \beta^2)\big (\norm{J^3(e^{\lambda \abs{x}} e^{\beta \abs{y}} w(0))}_{L^2} + \norm{J^3(e^{\lambda \abs{x}} e^{\beta \abs{y}} w(1))}_{L^2} \big)\\
									 + c\, \norm{e^{\lambda \abs{x}} e^{\beta \abs{y}} (\partial_t + \partial_x^3 + \partial_y^3 + a_1(\partial_x + \partial_y) + a_0)w}_{L_t^1 L_{x y}^2 \cap L_x^1 L_{y t}^2},
		\end{multline*}
		which yields the desired result.
	\end{proof}

\section{Persistence properties}\label{Appendix:persistence}
	In this section we are interested in studying persistence properties for solutions to~\eqref{Z-K_s}  
	
	In general a persistence property in a function space $X$ means that the solution $t\mapsto u(t)$ describes a continuous curve on $X,$ that is $u\in C([0,1]; X).$ 
	
	The theorem we are going to prove can be seen as a two dimensional generalization of the very well known result by Kato~\cite{Kato} for the KdV equation.

\begin{theorem}\label{persistence_properties}
	Let $u\in C([0,1]; H^4(\R^2)) \cap C^1([0,1]; L^2(\R^2))$ be a solution of the equation~\eqref{Z-K_s}.
	\begin{enumerate}[(i)]
		\item	If for all $\beta>0,$ $u(0)\in L^2(e^{2\beta(x+y)}dx\,dy),$ then $u$ is a bounded function from $[0,1]$ with values in $H^3(e^{2\beta(x+y)} dx\,dy)$ for all $\beta>0.$ 
		\item If for all $\beta>0,$ $u(1)\in L^2(e^{-2\beta(x+y)}\, dx dy),$ then $u$ is a bounded function from $[0,1]$ with values in $H^3(e^{-2\beta(x+y)} dx\,dy)$ for all $\beta>0.$
	\end{enumerate}
	
	\medskip
	In particular, if the conditions for $u(0)$ and $u(1)$ given in $(i)$ and $(ii),$ respectively, are satisfied, then $u$ is bounded from $[0,1]$ to $H^3(e^{2\beta \abs{x+y}}dx\,dy).$
\end{theorem}
The proof of Theorem~\ref{persistence_properties} is based on the following lemmas. The first lemma is an interpolation result that can be proved using the three-line theorem. 

\begin{lemma}\label{interpolation_persistence}
	For $s>0$ and $\beta>0$ let $f\in H^s(\R^2)\cap L^2(e^{2\beta(x+y)}\, dx dy).$ Then, for $\theta\in [0,1]$
	\begin{equation*}
		\norm{J^{\theta s}(e^{(1-\theta)\beta(x+y)} f)}_{L^2(\R^2)}\leq c \norm{J^s f}_{L^2(\R^2)}^{\theta} \norm{e^{\beta(x+y)} f}_{L^2(\R^2)}^{1-\theta}
	\end{equation*}
	where $J^s$ is such that $\widehat{J^s g}(\xi, \eta):= (1 + \xi^2 + \eta^2)^\frac{s}{2} \widehat{g}(\xi, \eta).$
\end{lemma}

In order to prove the exponential decay in Theorem~\ref{persistence_properties} we proceed in two steps, firstly we prove that $u$ is a bounded function from $[0,1]$ with values in $L^2(e^{2\beta(x+y)} dx\, dy)$ then, using the interpolation result Lemma~\ref{interpolation_persistence}, we obtain the boundedness of $u(t)$ in the space $H^3(e^{2\beta(x+y)} dx\, dy)$. The conclusion then follows from the symmetry properties of the equation.

The following lemma shows the boundedness of $u(t)$ in the space $L^2(e^{2\beta(x+y)} dx\, dy).$ The proof of this result follows mainly a strategy used in~\cite{B_I_M(JDE)} that came to light in the seminal paper by Kato~\cite{Kato} treating the well-posedness of the Cauchy problem for the KdV equation.

\begin{lemma}\label{partial_boundedness}
	Let $u\in C([0,1]; H^4(\R^2)) \cap C^1([0,1]; L^2(\R^2))$ be a solution of the equation~\eqref{Z-K_s} such that for all $\beta>0,$ $u(0)\in L^2(e^{2\beta(x+y)}dx\,dy).$ Then $u$ is a bounded function from $[0,1]$ with values in $L^2(e^{2\beta(x+y)} dx\,dy)$ for all $\beta>0.$ 
	\begin{proof}
		Since $e^{2\beta(x+y)}$ is a highly unbounded weight function, it is difficult to prove the result directly. Therefore we first approximate $e^{2\beta(x+y)}$ by a bounded weight function $\phi_n(x,y)$ which tends to $e^{2\beta(x+y)}$ monotonically as $n$ goes to infinity. Let $\varphi \in C^\infty(\R)$ be a decreasing function with $\varphi(x)=1$ if $x<1$ and $\varphi(x)=0$ if $x>10$ and let $\theta_n(x):= \int_0^x \varphi\big(\frac{x'}{n}\big)\, dx'.$ For $n\in \N$ we define
		\begin{equation*}
			\phi_n(x+y):=e^{2\beta \theta_n(x+y)}.
		\end{equation*}
 It can be seen that for every $n,$ $\phi_n(x+y)=e^{2\beta (x+y)}$ if $x+y\leq n$ and $\phi_n(x+y)\equiv d_n \leq e^{20\beta n}$ if $x+y>10n.$ Moreover $\phi_n\leq \phi_{n+1}$ and 
\begin{equation*}
	\abs{\partial_x^j \phi_n(x,y)}\leq C_{j,\beta} \phi_n(x,y),\qquad \abs{\partial_y^j \phi_n(x,y)}\leq C_{j,\beta} \phi_n(x,y) \qquad \forall\, j\in \N, \forall(x,y)\in \R^2.
\end{equation*}
Multiplying the equation~\eqref{Z-K_s} by $u\phi_n$ and integrating the resulting identity, we obtain
\begin{equation*}
	\int_{\R^2} \partial_t u\, u \phi_n + \int_{\R^2} \partial_x^3 u\, u \phi_n + \int_{\R^2} \partial_y^3 u\, u \phi_n + 4^{-1/3}\int_{\R^2} u^2 \partial_x u\, \phi_n + 4^{-1/3}\int_{\R^2} u^2 \partial_y u\, \phi_n =0.
\end{equation*}
Integrating by parts one has
\begin{equation*}
	\begin{split}
		\frac{1}{2} \frac{d}{dt} \int_{\R^2} u^2\, \phi_n &-\frac{1}{2} \int_{\R^2} u^2\, \partial_x^3 \phi_n +\frac{3}{2} \int_{\R^2} (\partial_x u)^2 \partial_x \phi_n -\frac{4^{-1/3}}{3}\int_{\R^2} u^3 \partial_x \phi_n\\
		&-\frac{1}{2} \int_{\R^2} u^2\, \partial_y^3 \phi_n +\frac{3}{2} \int_{\R^2} (\partial_y u)^2 \partial_y \phi_n -\frac{4^{-1/3}}{3}\int_{\R^2} u^3 \partial_y \phi_n=0,
	\end{split}
\end{equation*}
discarding positive terms this gives
\begin{equation*}
	\frac{1}{2} \frac{d}{dt} \int_{\R^2} u^2\, \phi_n\leq \frac{1}{2} \int_{\R^2} u^2\, \partial_x^3 \phi_n + \frac{1}{2} \int_{\R^2} u^2\, \partial_y^3 \phi_n + \frac{4^{-1/3}}{3}\int_{\R^2} u^3 \partial_x \phi_n + \frac{4^{-1/3}}{3}\int_{\R^2} u^3 \partial_y \phi_n.
\end{equation*}
Using the properties for the derivatives of $\phi_n$ and Sobolev embeddings one gets
\begin{equation*}
	\begin{split}
	\frac{1}{2} \frac{d}{dt} \int_{\R^2} u^2\, \phi_n &\leq C_{3,\beta} \int_{\R^2} u^2\, \phi_n + 2\frac{4^{-1/3}}{3} \norm{u(t)}_{L^\infty} C_{1,\beta} \int_{\R^2} u^2\, \phi_n\\
	&\leq \Big(C_{3,\beta} + C\norm{u}_{C([0,1]; H^2(\R^2))} \Big) \int_{\R^2} u^2\, \phi_n\\
	&=C_{\beta, u}\int_{\R^2} u^2\, \phi_n.
	\end{split}
\end{equation*} 
Applying the Gronwall lemma we obtain
\begin{equation*}
	\int_{\R^2} u^2(t) \phi_n\leq e^{C_{\beta, u} t}\int_{\R^2} u^2(0) \phi_n \leq e^{C_{\beta, u}}\int_{\R^2} u^2(0) \phi_n,\qquad \forall\, t\in [0,1]. 
\end{equation*}

Using the Monotone Convergence Theorem, letting $n$ go to infinity, we can conclude that	
\begin{equation*}
	\int_{\R^2} u^2(t) e^{2\beta(x+y)}\,dx\, dy\leq c \int_{\R^2} u^2(0) e^{2\beta(x+y)}\, dx\, dy,\qquad \forall\, t\in [0,1].
\end{equation*}
This proves that $u(t)$ is a bounded function from $[0,1]$ with values in $L^2(e^{2\beta(x+y)} dx\,dy)$ for all $\beta>0.$ 
\end{proof}
\end{lemma}  

\begin{proof}[Proof of Theorem~\ref{persistence_properties}]
\hfill

 \begin{enumerate}[(i)]
	\item We want to prove that, assuming $u(0)\in L^2(e^{2\beta(x+y)}dx\, dy)$ we have that $t\mapsto u(t)$ is bounded from $[0,1]$ with values in $H^3(e^{2\beta(x+y)} dx\, dy).$ In Lemma~\ref{partial_boundedness} we have already proved that if $u(0)\in L^2(e^{2\beta(x+y)}dx\, dy),$ then $u$ is a bounded function from $[0,1]$ with values in $L^2(e^{2\beta(x+y)}dx\, dy).$  Moreover since we are assuming $u\in C([0,1];H^4(\R^2))$ we can use the interpolation result Lemma~\ref{interpolation_persistence} with $s=4$ and $\theta=\frac{3}{2}$ to obtain
	\begin{equation*}
		\norm{J^3(e^{\frac{\beta}{4}(x+y)} u(t))}_{L^2(\R^2)}\leq c \norm{J^4 u(t)}_{L^2(\R^2)}^{\theta} \norm{e^{\beta(x+y)} u(t)}_{L^2(\R^2)}^{1-\theta}.
	\end{equation*}
	Since we are assuming the previous to hold for all $\beta>0,$ we can re-define $\beta$ in such a way to be able to conclude that $t\mapsto u(t)$ is bounded from $[0,1]$ with values in $H^3(e^{2\beta(x+y)}dx\, dy).$ 
	\item This property follows immediately from $(i)$ taking into account the symmetry properties of equation~\eqref{Z-K_s}. Indeed, it can be seen that the function defined as $\widetilde{u}(x,y,t):=u(-x,-y,1-t)$ is still a solution of~\eqref{Z-K_s}. Moreover, since we are assuming $u(1)\in L^2(e^{-2\beta(x+y)}dx\,dy)$ the function $\widetilde{u}$ satisfies the hypothesis of $(i),$ therefore $\widetilde{u}$ is bounded from $[0,1]$ with values in $H^3(e^{2\beta(x+y)} dx\, dy)$ for all $\beta>0,$ or, what is equivalent, $u$ is bounded from $[0,1]$ with values in $H^3(e^{-2\beta(x+y)}dx\,dy),$ which is the proof of $(ii).$ 
 \end{enumerate}
\end{proof}
	

\end{document}